\documentclass[a4paper,10pt,draft]{article}
\usepackage{amsfonts}
\usepackage{amsmath}
\usepackage{amssymb}
\usepackage{amsthm}
\usepackage[american]{babel}
\usepackage[T1]{fontenc}
\newtheorem{Theorem}{Theorem}[section]
\newtheorem{Definition}[Theorem]{Definition}
\newtheorem{Proposition}[Theorem]{Proposition}
\newtheorem{Lemma}[Theorem]{Lemma}

\newtheorem{Remark}[Theorem]{Remark}
\newtheorem{Example}[Theorem]{Example}
\newtheorem{Assumption}[Theorem]{Assumption}

\newcommand{\nd}{\stackrel{\textrm{def}}{=}}
\newcommand{\N}{{\mathbb N}}
\newcommand{\R}{{\mathbb R}}
\newcommand{\ud}{\mathrm{d}}
\newcommand{\eps}{\varepsilon}
\newcommand{\Lc}{{\cal L}}
\def\Swiech
{{\accent"13S}wie{\hbox{\kern -0.21em\lower
0.79ex\hbox{$\textfont1=\scriptfont1
\lhook$}}}ch}

\def\qedo{\hbox{\hskip 6pt\vrule width6pt height7pt
depth1pt  \hskip1pt}\bigskip}

\def\<{\left\langle }
\def\>{\right\rangle }
\newcommand{\myref}[1]{(\ref {#1})}

\def\cald{{\cal D}}

\def\calh{{\cal H}}

\author{
P. Acquistapace\footnote{Dipartimento di Matematica,
Universit\'a di Pisa, Italy, e-mail: paolo.acquistapace@unipi.it},
\; F.Gozzi\footnote {Dipartimento di Economia e Finanza,
Universit\`a \emph{LUISS - Guido Carli} Roma;
e-mail: fgozzi@luiss.it}}
\title{Minimum energy for linear systems with finite horizon:
\\ a non-standard Riccati equation}
\vspace{-1.5truecm}
\date{\today}
\begin{document}
\maketitle \vspace{-1truecm}
\begin{abstract}
This paper deals with a non-standard infinite dimensional
linear-quadratic control problem arising in the physics of non-stationary states
(see e.g. \cite{BDGJL4}): finding the minimum energy to drive a fixed stationary state $\bar x=0$ into an arbitrary non-stationary state $x$.
The Riccati Equation (RE) associated to this problem is not standard since
the sign of the linear part is opposite to the usual one, thus preventing the
use of the known theory.

Here we consider the finite horizon case.
We prove that the linear selfadjoint operator $P(t)$, associated to the
value function, solves the above mentioned RE (Theorem \ref{PsoldiR}).
Uniqueness does not hold in general but we are able to prove a partial uniqueness result in the class of invertible operators (Theorem \ref{unicitaRicc}).
In the special case
where the involved operators commute,
a more detailed analysis of the set of solutions is given (Theorems \ref{th:commuting1}, \ref{th:commuting2} and \ref{th:commuting3}).
Examples of applications are given.


\vskip 0.15cm

\vskip 0.15cm

\end{abstract}
\vspace{-0.5truecm}

\tableofcontents

\section{Introduction}\label{INTRO}

This paper is devoted to the study of a family of non-standard linear quadratic
finite horizon minimum energy problems in Hilbert spaces:
finding the minimum energy to drive a dynamical system from a fixed
equilibrium state $0$ (at time $t=t_0$) into an arbitrary
non-equilibrium state $x$ (at time $t=t_1$).
These problems arise (in particular when $t_0 \to -\infty$ and $t_1=0$) in the control representation of the rate function for a class of large deviation problems (see e.g. \cite{DaPratoPritchardZabczyk91} and the references quoted therein; see also \cite[Chapter 8]{FengKurtzbook} for an introduction to the subject); it is motivated by applications in the physics of non-equilibrium states and in this context it has been studied in various papers, see e.g. \cite{BDGJL1,BDGJL2,BDGJL3,BDGJL4,BDGJL5}.

In such papers the state equation is possibly nonlinear and the energy function
can be state dependent. One of the main goals, formulated e.g. in \cite{BDGJL4} in the infinite horizon case,
is then to show that the value function is the unique (or maximal/minimal) solution of the associated Hamilton-Jacobi-Bellman (HJB) equation.
Our goal is exactly this one. Due to the difficulty of the problem we restrict ourselves to study the linear quadratic case: hence solving the HJB equation reduces to solve a Riccati Equation (RE). In this paper, as a first step, we consider the finite horizon problem which we describe in the next subsection together with our main results.

\subsection{The problem and the main results}

To better clarify our results we state, roughly and informally, the mathematical problem
(see Subsection \ref{SSE:GENFORMUL} for a precise description).
The state space $X$ and the control space $U$ are both real separable Hilbert spaces.
We take the linear controlled system in $X$
\begin{equation}\label{storinf}
\left\{ \begin{array}{ll} y'(s)=Ay(s)+Bu(s), \quad s\in [0,t], \\[2mm]
y(0) = 0, \end{array}\right.
\end{equation}
where $A: D(A) \subset X \to X$ generates a strongly continuous semigroup
and $B:U\to X$ is a linear bounded operator.
Given a point $x \in X$ we consider the set ${\cal U}_{[0,t]}(0,x) $ of all control strategies $u(\cdot)$ that drive the system from the equilibrium state $0$
(at time $s=0$) into an arbitrary non-equilibrium state $x$ (at time
$s=t$). It is well known (see Subsection \ref{SSE:GENFORMUL}) that the set ${\cal U}_{[0,t]}(0,x) $ is nonempty if and only if $x\in H$, where $H$ is a suitable subspace of $X$ that can be endowed with its own Hilbert structure (see Subsection \ref{SSE:SPACEH}).

We want to minimize the ``energy-like'' cost functional
\begin{equation}\label{fzorinf}
J_{[0,t]}(u) =\frac12 \int_{0}^t \|u(s)\|^2_U\, \ud s.
\end{equation}
The value function $V(t,x)$ is defined as
\begin{equation}\label{eq:defV}
V(t,x) = \inf_{u\in {\cal U}_{[0,t]}(0,x)} J_{[0,t]}(u),
\end{equation}
and it is finite only when $x \in H$.
As the problem is linear quadratic, $V$ is a quadratic form in the variable $x$, i.e. $V(t,x)= \langle R(t)x, x\rangle_X$ for some symmetric operator-valued function $R(\cdot)$.
Hence we can consider the associated Riccati Equation (RE) in $X$ (with unknown $R(\cdot)$) which is, formally,
\begin{equation}\label{eq:algriccintroX}
\frac{d }{d t}\langle R(t)x,y\rangle_X=-\langle Ax, R(t)y\rangle_X - \langle R(t)x, Ay\rangle_X -\langle
B^*R(t) x, B^*R(t)y \rangle_U
\end{equation}
for every $x,y \in D(A)\cap D(R(t))$, with the initial condition $R(0)=+\infty$.
Since for each $t$ the operator $R(t)$ is unbounded (because $V(t,\cdot)$ is defined only in $H$), it is convenient to rewrite (\ref{eq:algriccintroX}) in $H$ so that the unknown $P(\cdot)$ becomes a bounded operator, see Subsection \ref{SS:VSOLVESRICCATI} for explanations.

Note that the sign of the linear part of (\ref{eq:algriccintroX})
(the first two terms of the right hand side)
is opposite to the usual one (see, e.g., for minimum energy problem in Hilbert spaces, \cite{DaPratoPritchardZabczyk91},
\cite{Emir89,Emir93}, \cite{GozziLoreti99}, \cite{PriolaZabczyk03}, \cite{Zabczyk92}).
This does not allow us to approach (\ref{eq:algriccintroX}) using the standard method (described e.g. in \cite[pp. 390-394 and 479-486]{BDDM07}, see also
\cite[p.1018]{PriolaZabczyk03}), which consists in
solving the RE using a fixed point theorem and a suitable {\it a priori} estimate. For forward RE like ours this is possible when the sign of the linear part is positive (in order to get a suitable semigroup generation property\footnote{More precisely in such case the linear part generates a semigroup (namely $P\mapsto e^{tA}Pe^{tA^*}$) which is not a group: such semigroup property is then lost when the sign changes.})
and the quadratic term is negative (in order to get the {\it a priori} estimate).


On the other hand the opposite sign of the linear part comes from the nature of the
motivating problem: to look at the minimum energy path from equilibrium to non-equilibrium states (see \cite{BDGJL4}), which is the opposite direction of the standard one considered e.g. in \cite{Curtain84,CurtainPritchard76,PriolaZabczyk03,Willems71},
(see also the books \cite{BDDM07,CurtainPritchard78,CurtainZwart95}).
This means that the value function depends on the final point, while
in the above quoted problems it depends on the initial one (see also Remark \ref{rm:timereversed} on this).
Therefore we are driven to use a different approach, that exploits
the structure of the problem; we partially borrow some ideas from \cite{PriolaZabczyk03}
and from\footnote{We thank prof. R. Vinter for providing us these references.} \cite{Moore81} and \cite{Scherpen93}. The main idea comes from the fact that the candidate solution of the RE (associated to the value function $V$ is the pseudoinverse of the unique solution of a Lyapunov (linear) equation (which is easier and is studied in Section \ref{SE:LYAPUNOV} providing an existence and uniqueness result in Proposition \ref{qtlyap}).

We list now our main results. We show, under a null controllability assumption, that the value function solves the associated Riccati Equation (RE) (Theorem \ref{PsoldiR}) and that a partial uniqueness holds (Theorem \ref{unicitaRicc}). When $A$ is selfadjoint and $A$ and $BB^*$ commute we can go deeper, finding more insights on the structure of the family of solutions (Theorems \ref{th:commuting1}, \ref{th:commuting2} and \ref{th:commuting3}).

\subsection{Plan of the paper}

Section \ref{sub:finitehorizon} is devoted to the presentation of our finite horizon minimum energy problem: after the description of our assumptions  (Subsection \ref{assumptions}) we provide the general formulation of the problem is in Subsection \ref{SSE:GENFORMUL}.

Section \ref{SE:LYAPUNOV} is devoted to the study of the associated Lyapunov equation, a key tool for the analysis of our RE. The main result of this section (Proposition \ref{qtlyap}) is more general than what we found in the literature and is then completely proved.

Section \ref{SE:RICCATI} is devoted to the analysis of the RE and to the presentation of the main results. It is divided in 5 subsections.
\begin{itemize}
\vspace{-0.2cm}
\item In the first (Subsection \ref{SSE:SPACEH}), we study the properties of the space $H$ which seems the good one where to study the RE.
\vspace{-0.2cm}
  \item Subsection \ref{SSE:VF} concerns the study the regularity properties of $V$.
\vspace{-0.2cm}
  \item In Subsection \ref{SS:VSOLVESRICCATI} we prove that $V$ solves the RE (Theorem \ref{PsoldiR}).
\vspace{-0.2cm}
  \item In Subsection \ref{SSE:PARTIALUNIQ} we present our partial uniqueness result (Theorem \ref{unicitaRicc}).
\vspace{-0.2cm}
  \item In Subsection \ref{SSE:COMMUTING} we refine our results in the special case of selfadjoint commuting operators.
\end{itemize}
Finally, Section \ref{examples} contains two significant examples.
At the end there is an Appendix divided in 4 parts. In the first three where we collect some preliminary results on pseudoinverses (\ref{SS:PSEUDOINVERSES}) on commuting operators (\ref{SSE:COMMOPER}), and
controllability operators (\ref{SSE:CONTROLLABOPERATORS}).
In the last one we collect the proofs of several lemmas and propositions (\ref{SSE:PROOFS}).

\section{Minimum energy problems}
\label{sub:finitehorizon}

\subsection{Assumptions}\label{assumptions}

Let $-\infty<s<t<+\infty$. Consider the abstract linear equation
\begin{equation}
\label{eq:state-fin-new} \left\{
\begin{array}{l}
y'(r)=Ay(r)+Bu(r), \quad r\in \, ]s,t] \\[2mm]
y(s) = z \in X,
\end{array}
\right.
\end{equation}
under the following assumption.

\begin{Assumption}\label{hp:main}
\begin{description}
\item[]
\item[(i)] $X$, the state space, and $U$, the control space, are real separable Hilbert spaces;
\item[(ii)] $A:D(A)\subseteq X \to X$ is the generator of a $C_0$-semigroup of negative type $-\omega$ in $X$ ($\omega>0$), i.e. there exists $M>0$ such that
\begin{equation}
\label{eq:AtipoMomega} \| e^{t A}\|_{{\cal L}(X)} \leq M e^{-\omega t}
\qquad \forall t\ge 0.
\end{equation}
\item[(iii)] $B\in {\cal L}(U,X)$, where ${\cal L}(U,X)$ is the space of bounded linear operators from $U$ to $X$;
\item[(iv)] $u$, the control strategy, belongs to $L^2(s,t;U)$.
\end{description}
\end{Assumption}

We recall the following well known result (see e.g. \cite[p. 106, Corollary 2.2 and Definition 2.3]{Pazy83}).

\begin{Proposition}
\label{prop:solcontinua} For $-\infty<s<t<+\infty$, $z\in X$ and
$u\in L^2(s,t;U)$, the mild solution of (\ref{eq:state-fin-new}),
defined by
\begin{equation}\label{vardcost}
y(r;s,z,u) = e^{(r-s)A}z + \int_s^r e^{(r-\sigma)A}Bu(\sigma) \, \ud \sigma, \quad
r\in [s,t],
\end{equation}
is in $C([s,t],X)$.
\end{Proposition}

In the sequel we will {\em always} assume that Assumption \ref{hp:main}
holds. Moreover, to prove most of the results of the paper we will also need
the assumption below. We state it now and we will say explicitly when we will use it. Before all we need to define the so-called controllability operator.

\begin{Definition}\label{df:Qt}
For $t\ge 0$ set
\begin{equation}\label{qt}
Q_t z = \int_0^t e^{rA} B B^* e^{rA^*}z \, \ud r, \qquad z\in X,
\end{equation}
and, for $t=+\infty$,
\begin{equation}\label{qinfty}
Q_\infty z = \lim_{t \to +\infty} Q_t z=
\int_0^{+\infty} e^{rA} B B^* e^{rA^*}z \, \ud r, \qquad z\in X.
\end{equation}
Note that $Q_\infty $ is well defined by Assumption \ref{hp:main}-(ii).
\end{Definition}

\begin{Assumption}\label{NC}
There exists $T_0>0$ such that\footnote{From now on we will denote by $R(F)$ the image of the operator $F$.}
\begin{equation}\label{eq:null-control}
R(e^{T_0A})\subseteq R(Q_{T_0}^{1/2}),
\end{equation}
\end{Assumption}
It is well known (see e.g. \cite[Appendix D]{DaPratoZabczyk14}) that this assumption is equivalent to assume null controllability at time $T_0$ for the system (\ref{eq:state-fin-new}) below: this means that for each $z\in X$ there exists a control $u\in L^2(0,T_0;U)$ such that the solution of (\ref{eq:state-fin-new}) with $[s,t]=[0,T_0]$ vanishes at time $T_0$.

\begin{Remark} \label{noexpdecay}
{\rm We have supposed in Assumption \ref{hp:main} that the semigroup $\{e^{tA}\}$ has negative type: this allows us to obtain more accurate results, also in view of a future study of the infinite horizon case. Anyway, if we only assume that $\|e^{tA}\|_{{\cal L}(X)}\le Me^{\gamma t}$ with $\gamma \ge 0$, most results of this paper are still true with suitable modifications. More specifically, since the operator $Q_\infty$ is not well defined, one has the following:
\begin{itemize}
\vspace{-0.3cm}
\item the space $H$ changes from $H=R(Q_\infty^{1/2})$ to $H=R(Q_T^{1/2})$, for suitable large $T>0$.
\vspace{-0.3cm}
\item Proposition \ref{valfun} modifies as follows:
\begin{itemize}
\vspace{-0.2cm}
\item (i) and (ii) hold in $\,[T_0,T]$;
\vspace{-0.1cm}
\item (iii)(a) holds in $\,[T_0,T]$;
\vspace{-0.1cm}
\item (iii)(b) holds in $\,]T_0,T]\times H$;
\vspace{-0.1cm}
\item (iii)(c) does not hold.
\end{itemize}
\vspace{-0.3cm}
\item The differential Riccati equation \eqref{RiccatiHprima} holds in $[T_0,T]$ with $Q_\infty$ replaced by $Q_T$.
\end{itemize}
Note that, in the commuting case, the proof of most results does not work as it is. This is the case for Proposition \ref{monotQt}(iii), Theorem \ref{th:commuting1} and Theorem \ref{th:commuting2}.
}\end{Remark}

\subsection{General formulation}
\label{SSE:GENFORMUL}

Suppose that Assumption \ref{hp:main} holds. Given a time interval $[s,t]\subset \R$, an initial state $z\in X$ and a control $u\in L^2(s,t;U)$ we consider the state equation \eqref{eq:state-fin-new}
and its mild solution $y(\cdot; s,x,u)$, given by \eqref{vardcost}.
We define the class of controls $u(\cdot)$ bringing the state $y(\cdot)$ from a
fixed $z\in X$ at time $s$ to a given target $x\in X$ at time $t$:
\begin{equation}
\label{eq:contr-x-x0} {\cal U}_{[s,t]}(z,x) \nd \left\{ u\in
L^2(s,t;U) \; : \; y(t;s,z,u)=x \right \}.
\end{equation}
We recall our cost functional, namely the energy:
\begin{equation}\label{eq:energyfunctional}
J_{[s,t]}(u) = \frac12 \int_s^t \|u(r)\|_U^2\, \ud r.
\end{equation}
The minimum energy problem at $(s,t;z,x)$ is the problem of
minimizing the functional $J_{[s,t]}(u)$ over all $u \in {\cal
U}_{[s,t]}(z,x)$. The value function of this control problem (the
{\em minimum energy}) is
\begin{equation} \label{eq:valuefunction-new} V_1(s,t;z,x)\nd
\inf_{u\in {\cal U}_{[s,t]}(z,x)} J_{[s,t]}(u).
\end{equation}
with the agreement that the infimum over the empty set is $+\infty$.
The following easy proposition, straightforward consequence of (\ref{vardcost}), allows to reduce the number of variables.
\begin{Proposition}\label{pr:V1lessvariables}
Under Assumption \ref{hp:main} we have
\begin{equation}\label{eq:spacontr-new}
\begin{array}{lcl} u(\cdot)\in {\cal U}_{[s,t]}(z,x)& \iff &
u(\cdot +t)\in  {\cal U}_{[s-t,0]}(0,x-e^{(t-s)A}z) \\[2mm]
& \iff & u(\cdot+s)\in  {\cal U}_{[0,t-s]}(0,x-e^{(t-s)A}z).\end{array}
\end{equation}
and then \\[2mm]
\rule{20mm}{0mm} $V_1(s,t;z,x)=V_1(s-t,0;0,x-e^{(t-s)A}z) =V_1(0,t-s;0,x-e^{(t-s)A}z)$. \hfill \qedo
\end{Proposition}
From now on we will set, for simplicity of notation,
\begin{equation}\label{eq:defV0-new}
V(t,x) := V_1(0,t;0,x) = \inf_{u\in {\cal U}_{[0,t]}(0,x)}
J_{[0,t]}(u) \qquad \forall t\in \,]0,+\infty[\,,\quad \forall x\in X.
\end{equation}
Now we look at the set where $V$ is finite: this is the {\em reachable set} in the interval
$[0,t]$, starting from $0$, defined as
\begin{equation}\label{eq:ptiragg-new}
{\cal R}_{[0,t]}^0:= \left\{ x \in X:\ {\cal U}_{[0,t]}(0,x) \neq \emptyset \right\}.
\end{equation}
Defining the operator
$${\cal L}_{t}:L^2(0,t;U) \to X, \qquad {\cal L}_{t}u=\int_0^{t}
e^{(t-\tau)A} B u(\tau)\, \ud \tau,$$
it is clear that
\begin{equation}\label{eq:ptiragg-newbis}
{\cal R}_{[0,t]}^0:= {\cal L}_{t}\left(L^2(0,t;U)\right),
\end{equation}
hence the set where $V$ is finite is $R\left( {\cal L}_{t}\right)$.


We now recall a fundamental, and well known, result, which establishes the relationship between the family of operators $\{Q_t, \;t\in [0,+\infty]\}$ and our minimum energy problem (see e.g. \cite[Theorem 2.3, p.210]{Zabczyk92}).
\begin{Theorem}\label{th:th23Zab}
Suppose that Assumption \ref{hp:main} holds and let $x\in X$.
\begin{itemize}
\item[(i)] The set ${\cal U}_{[0,t]}(0,x)$ is nonempty if and only if $x \in R({Q}^{1/2}_{t})$. In particular we have
\begin{equation}
\label{eq:ptiragg-Q12} {\cal L}_{t}\left(L^2(0,t;U)\right)={\cal
R}^0_{[0,t]}=R(Q_{t}^{1/2}) \qquad \forall t\ge 0\,.
\end{equation}
\item[(ii)] If $x \in R({Q}^{1/2}_{t})$, there is exactly one minimizing strategy
$\hat u_{t,x}$ for the functional $J_{[0,t]}$ over ${\cal U}_{[0,t]}(0,x)$, and moreover
\begin{equation}\label{valuef}
V(t,x) =J_{[0,t]}(\hat u_{t,x}) =\frac12 \|Q^{-1/2}_t x\|_X^2,
\end{equation}
where, for $t>0$, ${ Q}^{-1/2}_t: R(Q_t^{1/2})\to [\ker { Q}_t^{1/2}]^\perp $ is the pseudoinverse of $Q^{1/2}_t$.
\item[(iii)] If $x \in R(Q_t)$ then $V(t,x) =\frac12 \langle Q^{-1}_t x,x \rangle_X$, where $Q_t^{-1}: R({Q}_t)\to [\ker {Q}_t]^\perp$ is the pseudoinverse of $Q_t$.
\end{itemize}
\end{Theorem}
Since $V$ is quadratic the HJB equation associated to our problem becomes a differential Riccati Equation, namely \eqref{eq:algriccintroX}. Our main aim is then to prove that the linear symmetric operator $R$ associated to $V$ is a solution of such Riccati Equation and prove a kind of uniqueness result. We will do this in Section \ref{SE:RICCATI}.

\begin{Remark}\label{rm:infinitehorizonbefore}
{\em It is possible to extend the above minimum energy problem to the case when
$s=-\infty$ or when $t=+\infty$. The energy functional becomes then an integral over a half line. In the first case we have to take the initial datum $z=0$ and, properly defining the mild solutions in the left half-line (requiring that \myref{vardcost} is satisfied for all $r\ge s =-\infty$),
we have to define the set of control strategies as follows:
\begin{equation*}
{\cal U}_{[-\infty,t]}(0,x) \nd \left\{ u\in
L^2(-\infty,t;U) \; : \; y(t;-\infty,0,u)=x \right \}.
\end{equation*}

In the second case the problem is trivial. Indeed formally one should define
\begin{equation*}
{\cal U}_{[s,+\infty]}(z,x) \nd \left\{ u\in
L^2(s,+\infty;U) \; : \; \lim_{t\to + \infty} y(t;s,z,u)=x \right \}.
\end{equation*}
However it is easy to show that for every $u\in L^2(s,+\infty;U)$ we have
$\lim_{t\to + \infty} y(t;s,z,u)=0$, so that the class ${\cal U}_{[s,+\infty]}(z,x)$
is empty unless $x=0$; in this case the optimal control strategy is clearly $u \equiv 0$.

In a subsequent paper we will study the infinite horizon problem when the
starting time is $- \infty$ and the arrival time is $0$: the value
function of this problem is formally $V_1(-\infty,0;0,x)=V(+\infty, x)$.
Some results about it will be also given in the present paper.
For simplicity we will use the notation
$$
V_\infty (x):=V_1(-\infty,0;0,x).
$$
In Proposition \ref{valfun} we will prove that, under Assumption \ref{hp:main}, we have
$$V_\infty(x) = \lim_{t\to +\infty} V(t,x)=\inf_{t>0} V(t,x).\hfill\qedo$$}
\end{Remark}

\section{The Lyapunov equation}
\label{SE:LYAPUNOV}

We want now to show that the function $t\to Q_t$, from $[0,+\infty)$ to $\Lc(X)$, solves
a suitable Lyapunov equation. To this purpose we prove first the following lemma.

\begin{Lemma}\label{proreg}
\begin{itemize}
\item[]
\item[{\bf (i)}]
If $x\in D(A^*)$, then
for every $t \in [0,+\infty]$ we have $x \in D(AQ_t)$ and
\begin{equation}\label{Lya1}
AQ_t x = e^{tA}BB^* e^{tA^*}x - BB^*x - Q_tA^*x \qquad \forall x\in D(A^*) \quad \forall t \in[0,+\infty[\,,
\end{equation}
\begin{equation}\label{Lya1bis}
AQ_\infty x =  - BB^*x - Q_\infty A^*x \qquad \forall x\in D(A^*).
\end{equation}
\item[{\bf (ii)}]
For every $t\in [0,+\infty]$ we have $D(A^*) \subseteq D((AQ_t)^*)\subseteq D(AQ_t)$, they all are dense in $X$, and
\begin{equation}\label{Lya}
(AQ_t) x = e^{tA}BB^* e^{tA^*}x - (AQ_t)^*x - BB^*x \quad \forall x\in D((AQ_t)^*) \quad \forall t \in[0,+\infty[\,.
\end{equation}
\begin{equation}\label{Lyainfty}
(AQ_\infty) x = - (AQ_\infty)^*x - BB^*x \quad \forall x\in D((AQ_\infty)^*).
\end{equation}
\item[{\bf (iii)}]
For  every $t\in \,]0,+\infty]$,
if $x\in Q_t(D((AQ_t)^*))$, then $Ax \in [\ker Q_t]^\perp$.
\item[{\bf (iv)}]
For  every $t\in \,]0,+\infty]$, $Q_t(D(A^*))$ is dense in $[\ker Q_t]^\perp$.
Hence if $x\in D(A) \cap [\ker Q_t]^\perp$, then $Ax \in [\ker Q_t]^\perp$.
\end{itemize}
\end{Lemma}
\begin{proof} {\bf (i)} Let $x\in D(A^*)$. Then we can write, integrating by parts:
\begin{eqnarray*}
Q_tx & = & \int_0^t e^{rA}BB^*e^{rA^*}x\,dr \\
& = & A^{-1} \left[ e^{rA}BB^* e^{rA^*}x \right]_0^t - A^{-1}\int_0^t e^{rA}BB^*e^{rA^*}A^*x\,dr \\
& = & A^{-1} \left[e^{tA}BB^* e^{tA^*}x - BB^*x - Q_t A^*x\right],
\end{eqnarray*}
and (i) follows.\\[1mm]
{\bf (ii)} The first inclusion follows from the very definition of the adjoint.


Next, if $x\in D((AQ_t)^*)$ we can write for each $y\in D(A^*)$, by (\ref{Lya1}),
\begin{eqnarray*}
|\langle Q_tx,A^*y\rangle_X| & = & |\langle x,Q_tA^*y\rangle_X| \\
& = & |\langle x, -AQ_ty + e^{tA}BB^* e^{tA^*}y - BB^*y\rangle_X| \\
& = & |\langle -(AQ_t)^*x + e^{tA}BB^* e^{tA^*}x - BB^*x, y\rangle_X| \le c_t \|y\|_X\,,
\end{eqnarray*}
so that $Q_tx\in D(A)$, i.e. $x\in D(AQ_t)$, and (\ref{Lya}) holds. This proves the claim
for $t\in[0,+\infty[$. For the case $t=+\infty$ we argue in a similar way: let $x\in D((AQ_\infty)^*)$; then for each $y\in D(A^*)$, by (\ref{Lya1bis}) we have,
\begin{eqnarray*}
|\langle Q_\infty x,A^*y\rangle_X| & = & |\langle x,Q_\infty A^*y\rangle_X| \\
& = & |\langle x, -AQ_\infty y - BB^*y\rangle_X| \\
& = & |\langle -(AQ_\infty )^*x - BB^*x, y\rangle_X| \le c_t \|y\|_X\,,
\end{eqnarray*}
so that $Q_\infty x\in D(A)$, i.e. $x\in D(AQ_\infty )$, and (\ref{Lyainfty}) holds.
\\[1mm]
{\bf (iii)}
By assumption we have $x=Q_t z$ with $z\in D((AQ_t)^*)$. Let $w\in \ker Q_t$.
By Proposition \ref{pr:Qt} and its proof we get $B^*e^{sA^*}w=0$ for all $s\in [0,t]$. Moreover $w$ belongs obviously to $D(AQ_t)$ (with $AQ_t w =A\,0 =0$).
Hence we have by \myref{Lya}, when $t<+\infty$
$$\langle Ax,w\rangle_X = \langle AQ_t z, w\rangle_X =
\langle B^*e^{tA^*}z,B^*e^{tA^*}w\rangle_X
-\langle B^*z,B^*w\rangle_X - \langle z,AQ_\infty w\rangle_X = 0.$$
As a consequence, $Ax\in [\ker Q_\infty]^\perp$.
Similarly, if $t=+\infty$, we have for every $w\in \ker Q_\infty$, by \myref{Lyainfty},
$$\langle Ax,w\rangle_X = \langle AQ_\infty z, w\rangle_X = -\langle B^*z,B^*w\rangle_X - \langle z,AQ_\infty w\rangle_X = 0,$$
and the claim follows.
\\[1mm]
{\bf (iv)}
We just consider the case $t=\infty$, since the case $0<t <\infty$ is quite similar.
Fix $x\in [\ker Q_\infty]^\perp$. As $[\ker Q_\infty]^\perp = \overline{R(Q_\infty)}$
there is a sequence of elements $x_n\in R(Q_\infty)$ such that $x_n \to x$. Hence
there exists $\{z_n\}\subset X$ such that $Q_\infty z_n \to x$ in $X$. Since $D(A^*)$ is dense in $X$, for each $n\in \N^+$ we can find $y_n \in D(A^*)$ such that $\|y_n-z_n\|_X<1/n$, so that $Q_\infty y_n \to x$ in $X$, too.

For the last statement, observe first that, since $D(A^*) \subseteq D((AQ_\infty)^*)$, then
$Q_\infty(D(A^*)) \subseteq Q_\infty(D((AQ_\infty)^*))$, so the latter is dense in $[\ker Q_\infty]^\perp$, too.
Hence, for $x\in D(A) \cap [\ker Q_\infty]^\perp$, let $\{z_k\}_{k\in \N}\subset D((AQ_\infty)^*)$ be a sequence such that
$Q_\infty z_k \to x$ in $X$ as $k\to \infty$. Thus we can write for $w\in \ker Q_\infty$
\begin{eqnarray*}
\langle Ax,w\rangle_X & = & \lim_{k\to \infty} \langle AQ_\infty z_k,w\rangle_X = \lim_{k\to \infty} \langle -(AQ_\infty)^*z_k-BB^*z_k,w\rangle_X \\[2mm]
& = & \lim_{k\to \infty}  -\langle z_k,AQ_\infty w\rangle_X - \langle B^*z_k,B^*w\rangle_X= \lim_{k\to \infty} 0 = 0.
\end{eqnarray*}
\end{proof}


\begin{Definition}\label{defqtlyap}
A map $Q(\cdot):[0,+\infty) \to \mathcal{L}(X)$ is a solution of
the differential Lyapunov equation
\begin{equation}\label{eq:Lyapdiff}
\left\{ \begin{array}{l} Q'(t) = AQ(t) + Q(t)A^* + BB^*, \quad t>0,\\
Q(0)=0, \end{array} \right.
\vspace{-0.2cm}
\end{equation}
\vspace{-0.2cm}
if:
\begin{itemize}
  \item for each $t\ge 0$ the operator $Q(t)$ is positive and selfadjoint
  and $Q(0)=0$;
\vspace{-0.2cm}
  \item for each $t\ge 0$ and $x \in D(A^*)$ we have $Q_t x \in D(A)$;
\vspace{-0.2cm}
  \item for each $x \in D(A^*)$ the map $t\to Q(t)x$ is differentiable and
$$\frac{\ud}{\ud t}Q_tx = AQ_tx + Q_tA^*x + BB^*x \qquad \forall t>0.$$
\end{itemize}
Similarly an operator $Q\in \mathcal{L}(X)$ is a solution of
the algebraic Lyapunov equation
\begin{equation}\label{eq:LyapunovX}
 AQ + QA^* + BB^*=0
\vspace{-0.2cm}
\end{equation}
if:
\begin{itemize}
\vspace{-0.2cm}
  \item $Q$ is positive and selfadjoint;
\vspace{-0.2cm}
  \item for each $x \in D(A^*)$ we have $Q x \in D(A)$;
\vspace{-0.2cm}
  \item for each $x \in D(A^*)$
$$ AQx + QA^*x + BB^*x=0.$$
\end{itemize}
\end{Definition}

\begin{Proposition}
\label{qtlyap}
The operator $Q_t$ defined by (\ref{qt}) is
a solution of the differential Lyapunov equation (\ref{eq:Lyapdiff}).
Similarly the operator $Q_\infty$ solves the algebraic Lyapunov equation
(\ref{eq:LyapunovX}).

Moreover, for all $t\ge 0$, let $Q(t)$ be positive and selfadjoint, and such that it solves the Lyapunov equation \myref{eq:Lyapdiff} in weak sense, i.e. $Q(0)=0$,
the map $t\mapsto \<Q(t) x,y\>_X$ is differentiable for every $x,y \in D(A^*)$ and
$$\frac{\ud}{\ud t}\<Q(t)x,y\>_X = \<Q(t)x,A^*y\>_X +\< A^*x, Q(t)y\>_X + \<B^*x,B^*y\>_U \qquad \forall t>0.$$
Then $Q(t)=Q_t$ for all $t\ge 0$.

Similarly let $Q$ be  positive and selfadjoint, and such that it solves
the Lyapunov equation \myref{eq:LyapunovX} in weak sense, i.e.
for all $x,y \in D(A^*)$
$$ \<Qx,A^*y\>_X +\< A^*x, Qy\>_X + \<B^*x,B^*y\>_U =0 .$$
Then $Q=Q_\infty$.
\end{Proposition}
\begin{proof}
We give the proof for the reader's convenience since we did not find it in the literature. Indeed, in \cite[Theorem 5.1.3]{CurtainZwart95}, in \cite[part II, Chapter 1, Theorem 2.4]{BDDM07} and in \cite[Appendix D]{DaPratoZabczyk14} only the algebraic Riccati equation is considered, and it is shown that it has a positive operator-valued solution if and only if the semigroup generated by $A$ is exponentially stable.

Consider first the differential Lyapunov equation (\ref{eq:Lyapdiff}).
By the definition of $Q_t$ we obviously have
$$\frac{\ud}{\ud t}Q_tx = e^{tA}BB^* e^{tA^*} x \qquad \forall x\in X.$$
Then the existence result follows from Lemma \ref{proreg}-(i).

Concerning uniqueness, we observe that, if $Q_1(t)$ and $Q_2(t)$ are two functions with values in the space of bounded, selfadjoint, positive operators, and they both solve (\ref{eq:Lyapdiff}) in weak sense, then the difference $Q(t):=Q_1(t)-Q_2(t)$ satisfies the homogeneous equation
$$\frac{\ud}{\ud t}\<Q(t)x,y\>_X = \<Q(t)x,A^*y\>_X +\< A^*x, Q(t)y\>_X \qquad \forall t>0,$$
with $Q(0)=0$.
Now take any $x \in D(A^*)$ and $t_0>0$ and observe that, by simple computations,
$$
\frac{d}{dt}\langle Q(t)e^{(t_0-t)A^*}x,e^{(t_0-t)A^*}x  \rangle_X=0,
$$
so that it must be
$$
\langle  Q(t_0)x,x \rangle _X=
\langle  Q(0)e^{t_0 A^*}x,e^{t_0 A^*}x \rangle _X=0.
$$
Since $Q(t_0)$ is selfadjoint, we can use polarization to get $Q(t_0)=0$ for every $t_0>0$ and so the claim.

Now we look at the algebraic Lyapunov equation (\ref{eq:LyapunovX}).
From (\ref{Lya1bis}) it follows that $Q_\infty$ solves (\ref{eq:LyapunovX}).
To show uniqueness, similarly for the case of the differential Lyapunov equation, we
observe that, if $Q_1$ and $Q_2$ are two bounded, selfadjoint, positive operators which solve (\ref{eq:LyapunovX}) in weak sense, then the difference $Q:=Q_1-Q_2$ satisfies the homogeneous equation
$$ \<Qx,A^*y\>_X +\< A^*x, Qy\>_X =0 .$$
Hence, for any $x \in D(A^*)$ as before we deduce
$$
\frac{d}{dt}\langle Qe^{tA^*}x,e^{tA^*}x  \rangle_X=0,
$$
so that it must be, since $A$ is of negative type,
$$
\langle  Qx,x \rangle _X=\lim_{t \to + \infty}
\langle  Qe^{t A^*}x,e^{t A^*}x \rangle _X=0.
$$
As above, since $Q$ is selfadjoint, we use polarization getting $Q=0$ and so $Q_1=Q_2$.
\end{proof}

\begin{Remark}
{\rm
If $A$ is selfadjoint and commutes with $BB^*$, then, by Proposition \ref{pr:Qt}-(v), it also commutes with $Q_t$, $t \in [0,+\infty]$. Moreover, by the Lyapunov equation \myref{eq:Lyapdiff} we have, for all $x \in D(A)$,
$$
\frac{\ud}{\ud t}Q_t x=2AQ_t x + BB^*x
$$
and, by \myref{eq:LyapunovX}, we have, for all $x \in D(A)$,
$$
2A Q_\infty x =-BB^* x.
$$
Indeed, this last equality holds for all $x \in X$, as it follows from \myref{eq:Qtcommutingformula}.
Finally, from the last one we easily get, for all $y \in R(Q_\infty)\subseteq D(A)$,
taking $x=Q_\infty^{-1}y$,
$$
2A y =- BB^* Q_\infty^{-1}y.
$$
}
\hfill\qedo
\end{Remark}

\section{The Riccati equation}
\label{SE:RICCATI}

From Theorem \ref{th:th23Zab} above we know that the value function
$V(t,\cdot)$ is finite only in the set $R(Q_t^{1/2})$ and is given by
$V(t,x)=\frac12\|Q_t^{-1/2}x\|^2_X$. Moreover for points $x\in R(Q_t)$ we
can write $V(t,x)=\frac12\langle Q_t^{-1}x,x \rangle_X$. So $V$ is a
quadratic form on $X$, defined however only for $x\in R(Q_t^{1/2})$; thus we expect that the associated operator $Q_t^{-1}$ solves\footnote{See Definition \ref{df:solutionDRE} for the formal definition of solution.} our Riccati equation \eqref{eq:algriccintroX},
which we rewrite here for the reader's convenience:
\begin{equation}\label{eq:evriccprima}
\frac{\ud}{\ud t}\langle R(t)x,y \rangle_X = -\langle Ax, R(t)y\rangle_X - \langle R(t)x,Ay\rangle_X - \langle B^*R(t) x,B^*R(t)y \rangle_U\, , \; t>0,
\end{equation}
for every $x,y\in D(A)\cap D(R(t))$, with the initial condition $R(0^+)=+\infty$.
This is indeed the case, as we will prove later (see Theorem \ref{PsoldiR} below).
Note that the initial condition has to be properly interpreted and that we cannot expect uniqueness of the RE without any initial condition as, obviously, $R\equiv 0$ is a solution.



Equation (\ref{eq:evriccprima}) is hard for several reasons: the
infinite initial condition (arising also in \cite{PriolaZabczyk03}),
the negative sign of the linear part (which does not arise in
\cite{PriolaZabczyk03}) and the unboundedness of the expected
solution (which is also not present in
\cite{PriolaZabczyk03}). Indeed the difference due to the negative sign is
substantial: even in the simplest diagonal case (see Subsection \ref{SSE:DIAGONAL})
there is no semigroup associated to the linear part of (\ref{eq:evriccprima}) on the whole space $X$, so that the equation cannot be rewritten in mild form as usual (see e.g \cite[Theorem 4.1, p. 234]{Zabczyk92}).

Note that, if we change the sign of the linear part, then we are exactly in the case treated by \cite{PriolaZabczyk03}, and the solution, when the null controllability Assumption \ref{NC} holds, just coincides with the operator-valued function on $X$ given, formally, by $e^{tA^*}Q_t^{-1}e^{tA}$.

\begin{Remark}
\label{rm:timereversed} 
{\rm We observe that performing a time inversion in the state equation, or in the RE, does
not change the difficulty of the problem, which lies in the fact that the equation is forward and the linear part is negative. Of course this is not true if $A$ generates not just a $C_0$-semigroup but a $C_0$-group (this includes the case of bounded $A$).
We do not want to assume this, since our examples, in particular the diagonal one (which arises in our motivating application to physics, see \cite{BDGJL4} and Subsection \ref{SSE:DIAGONAL}) does not possess such property.
}
\end{Remark}

\subsection{The space $H$ and its properties}
\label{SSE:SPACEH}

In order to study equation (\ref{eq:evriccprima}) it will be useful to rewrite it in a different form and in a different space, which we call $H$: under the null controllability assumption (Assumption \ref{NC}) it is the reachable set of the control system \myref{eq:state-fin-new}, hence the set where the value function $V$ of \myref{eq:defV0-new}
is well defined. Then
we define
\begin{equation}\label{spH}
H=R(Q_\infty^{1/2}).
\end{equation}
Of course it holds
$$H \subseteq \overline{R(Q_\infty^{1/2})}= [\ker Q_\infty^{1/2}]^\perp
=[\ker Q_\infty]^\perp.
$$
The inclusion is in general proper.
Define in $H$ the inner product
\begin{equation}\label{nH}
\langle x, y\rangle_H = \langle Q_\infty^{-1/2}x,Q_\infty^{-1/2}y \rangle_X \qquad \forall x,y\in H.
\end{equation}

We provide now some useful results on the space $H$ which will form the ground for our main results. We divide them in six Lemmas, whose proofs are collected in Appendix \ref{SSE:PROOFS}.
The first three concern the structure of the space $H$ and the behaviour in $H$ of the operators $Q_t$.
\begin{Lemma}\label{HH}
\begin{itemize}
  \item[]
  \item[(i)] The space $H$ introduced in (\ref{spH}), endowed with the inner product (\ref{nH}), is a Hilbert space continuously embedded into $X$.

  \item[(ii)] The space $R(Q_\infty)$ is dense in $H$.

  \item[(iii)] The operator $Q_\infty^{-1/2}$ is an isometric isomorphism from $H$ to $[\ker Q_\infty^{1/2}]^\perp$, and in particular
\begin{equation}\label{qinfty2}
\|Q_\infty^{-1/2}x\|_X = \|x\|_H \qquad \forall x\in H.
\end{equation}
  \item[(iv)]
We have
$$
\|Q_\infty^{1/2}\|_{\Lc(X)}=\|Q_\infty^{1/2}\|_{\Lc(H)}.
$$
  \item[(v)] For every $F \in \Lc(X)$ such that $R(F) \subseteq H$ we have $Q_\infty^{-1/2}F \in \Lc (X)$, so that $F\in \Lc(X,H)$.
\end{itemize}
\end{Lemma}

\begin{Lemma}\label{dens} For $0<t\le \infty$ let $Q_t$ be the operator defined by (\ref{qt}). Then, if $t\in [T_0, + \infty]$ the space $Q_t(D(A^*))$ is dense in $H$.
In particular $D(A)\cap H$ is dense in $H$. Finally, if $x\in D(A)\cap H$
then $Ax \in [\ker Q_\infty]^\perp$.
\end{Lemma}
\begin{Lemma}\label{relinH} For $0<t\le \infty$ let $Q_t$ be the operator defined by (\ref{qt}). Then
\begin{description}
\item[(i)] $\langle z, Q_t^{-1/2}w\rangle_X = \langle Q_t^{-1/2}z, w\rangle_X$ for all $z,w\in R(Q_t^{1/2})$;
\item[(ii)] $\langle Q_\infty^{1/2}x, y\rangle_H = \langle x, Q_\infty^{1/2}y\rangle_H$ for all $x,y\in H$;
\item[(iii)] $\langle Q_\infty x, y\rangle_H = \langle x, Q_\infty y\rangle_H$ for all $x,y\in H$.
\end{description}
\end{Lemma}
Suppose now that Assumption \ref{NC} holds. The next two lemmas deal with the operators $Q_t^{1/2}Q_\infty^{-1/2}$ and $Q_t^{-1/2}Q_\infty^{1/2}$. By Proposition \ref{monotQt}, we have
$$H = R(Q_\infty^{1/2}) = R(Q_t^{1/2}) \qquad \forall t\ge T_0,$$
so that $Q_t^{1/2}Q_\infty^{-1/2}$ is well defined from $H$ into $H$.
\begin{Lemma}\label{composQt} Under Assumption \ref{NC}, for fixed $t\ge T_0$ the operator $Q_t^{1/2}Q_\infty^{-1/2}:H\to H$ is an isomorphism, with inverse $Q_\infty^{1/2}Q_t^{-1/2}$.
Similarly for fixed $t\ge T_0$ the operator $Q_t^{-1/2}Q_\infty^{1/2}:H\to H$ is an isomorphism, with inverse $Q_\infty^{-1/2}Q_t^{1/2}$.
\end{Lemma}
Similarly, we have:
\begin{Lemma}\label{composQt2} Under Assumption \ref{NC}, for fixed $t\ge T_0$ the operator $Q_t^{-1/2}Q_\infty^{1/2}:X\to X$ is an isomorphism on the closed subspace
$[\ker Q_\infty]^\perp= \overline{R(Q_\infty^{1/2})}=[\ker Q_\infty^{1/2}]^\perp$, with
$$[Q_\infty^{-1/2}Q_t^{1/2}]Q_t^{-1/2}Q_\infty^{1/2}x=P_{[\ker Q_\infty]^\perp}x \qquad \forall x\in X.$$
\end{Lemma}
The last lemma describes the adjoint in $H$ of an operator $L\in {\cal L}([\ker Q_\infty]^\perp)\cap{\cal L}(H)$.
\begin{Lemma}\label{aggH} Let $L\in {\cal L}([\ker Q_\infty]^\perp)\cap {\cal L}(H)$. Then
$$\langle Lx,y\rangle_H = \langle x, Q_\infty L^* Q_\infty^{-1}y\rangle_H  \qquad \forall x\in H, \quad \forall y\in R(Q_\infty),$$
where $L^*\in {\cal L}([\ker Q_\infty]^\perp)$ is the adjoint of the operator $L$ in $[\ker Q_\infty]^\perp$.
\end{Lemma}
To avoid confusion, for any $L\in {\cal L}(H)$ we will denote by $L^{*H}$ the adjoint of $L$ in $H$, i.e. $L^{*H}=Q_\infty L^* Q_\infty^{-1}$. Moreover, for a subspace $V$ of $H$ we will write $V^{*H}$ for the topological dual of $V$ when $H$ is identified with its dual.\\
We remark that, under Assumption \ref{NC}, if $y\in R(Q_\infty)= Q_\infty^{1/2}(H)$ we have $Q_\infty^{-1/2}y \in H$ and, by Lemma \ref{composQt}, $Q_\infty^{1/2}Q_t^{-1/2}Q_\infty^{-1/2}y \in H$. Thus $Q_\infty Q_t^{-1/2}Q_\infty^{-1/2}y \in Q_\infty^{1/2}(H)=R(Q_\infty)$.
Consequently, under Assumption \ref{NC} we may write, by Lemma \ref{aggH},
\begin{equation}\label{adjQQ}
[Q_\infty^{1/2}Q_t^{-1/2}]^{*H}y = Q_\infty [Q_\infty^{1/2}Q_t^{-1/2}]^*Q_\infty^{-1}y = Q_\infty Q_t^{-1/2}Q_\infty^{-1/2}y \qquad \forall y\in R(Q_\infty).
\end{equation}

\subsection{Properties of the value function}
\label{SSE:VF}

We now state the main properties of the value function $V(t,x)$ defined by (\ref{eq:defV0-new}).
The proofs are in Appendix \ref{SSE:PROOFS}.

\begin{Proposition}\label{valfun} The value function $V$ given by (\ref{eq:defV0-new}) has the following properties:
\begin{itemize}
  \item[(i)] For every $t_0 >0$ and $x \in R(Q_{t_0}^{1/2})$, the function $V(\cdot,x)$ is decreasing in $\,]t_0,+ \infty[\,$.
  \item[(ii)] For every $t>0$ the function $V(t,\cdot)$ is quadratic with respect to $x\in R(Q_t^{1/2})$, i.e. there exists a linear positive selfadjoint operator
  $$P_V(t): R(Q_{t}^{1/2})\subseteq H \to [R(Q_{t}^{1/2})]^{*H} \supseteq H$$
      such that
\begin{equation}\label{forquaV1}
V(t,x)=\frac12 \langle P_V(t) x,x \rangle_{[R(Q_{t}^{1/2})]^{*H},R(Q_{t}^{1/2})} \qquad \forall t>0, \quad \forall x\in R(Q_{t}^{1/2});
\end{equation}
moreover we have
\begin{equation}\label{formaP1}
P_V(t) = [Q_\infty^{1/2}Q_t^{-1/2}]^{*H}Q_\infty^{1/2}Q_t^{-1/2} \qquad \forall t>0.
\end{equation}
  \item[(iii)]
Assume now that Assumption \ref{NC} holds. Then:
\begin{itemize}
  \item[(a)]  the operator $P_V(t)$ belongs to $\mathcal{L}(H)$ and
\begin{equation}\label{forquaV2}
V(t,x)=\frac12 \langle P_V(t) x,x \rangle_H \qquad \forall t\ge T_0, \quad \forall x\in H;
\end{equation}
in particular,
\begin{equation}\label{formaP2}
P_V(t)x = Q_\infty Q_t^{-1}x \qquad \forall x\in R(Q_t), \quad \forall t\ge T_0.
\end{equation}
In addition
\begin{equation}\label{glolim}
\|P_V(\tau)\|_{{\cal L}(H)}\le \|P_V(t)\|_{{\cal L}(H)}\le \|P_V(T_0)\|_{{\cal L}(H)} < \infty \qquad \forall \tau \ge t\ge T_0.
\end{equation}
  \item[(b)]
The map $(t,x) \mapsto V(t,x)$ from $[T_0,+\infty[\,\times H$ to $\R$ is continuous, uniformly on $[T_0,+\infty[\,\times B_H(0,R)$ for every $R>0$; moreover the map $t\mapsto P_V(t)$ from $[T_0,+\infty[\,$ to ${\cal L}(H)$ is continuous.
\item[(c)] Finally, we have
\begin{equation}\label{Vinfty}
\lim_{t\to \infty} V(t,x)= \frac12 \|x\|_H^2 \qquad \forall x\in H.
\end{equation}
\end{itemize}
\end{itemize}
\end{Proposition}

\begin{Remark}\label{extP} {\em The equations (\ref{formaP1}) and (\ref{formaP2}) show that the operator $Q_\infty Q_t^{-1}$, defined on $R(Q_t)$, has in fact an extension to all of $H$, given by $[Q_\infty^{1/2}Q_t^{-1/2}]^{*H}Q_\infty^{1/2}Q_t^{-1/2}$, i.e. $P(t)$.}\hfill \qedo
\end{Remark}

\subsection{The value function solves the Riccati equation}
\label{SS:VSOLVESRICCATI}

We want now to show that the operator $P_V(t)$, given by (\ref{formaP1}) or (\ref{formaP2}), satisfies for $t\ge T_0$ the Riccati equation \myref{eq:evriccprima}. To do this we first rewrite it in the space $H$. The unknown is now, for all $t\in [0,T]$, an operator $P(t) \in \Lc(H)$ which is, formally, $Q_\infty R(t)$ where $R$ is the unknown of \myref{eq:evriccprima}, while the equation is
\begin{equation} \label{RiccatiHprima}
\frac{\ud}{\ud t}\langle P(t)x,y \rangle_H = -\langle Ax, Q_\infty^{-1}P(t)y\rangle_X - \langle Q_\infty^{-1}P(t)x,Ay\rangle_X - \langle B^*Q_\infty^{-1}P(t)x,B^*Q_\infty^{-1}P(t)y \rangle_U.
\end{equation}
Note that the term in the left-hand side is written using the inner product of the space $H$ while the first two in the right-hand side are written with the inner product in $X$:
they could be written in $H$, too, but at the price of requiring more regularity on the points $x,y$ (since $Ax, Ay$ in this case should belong to $H$).

\begin{Definition}
\label{df:solutionDRE}
Let $0<t_0<+\infty$.
\begin{itemize}
  \item[(i)] An operator-valued function $P:[t_0,+\infty[\,\to  \Lc_+(H)$ is a solution of the Riccati equation (\ref{RiccatiHprima}) if it is strongly continuous and for all $t\ge t_0$ there is a set $D_P(t) \subset H$, dense in $H$, such that for every $x,y\in D_P(t)$ there exists $\<\frac{\ud}{\ud t}P(t)x,y\>_H$, all terms of (\ref{RiccatiHprima}) make sense and the equation holds.

  \item[(ii)] A function $R$, defined on $[t_0,+\infty[\,$ with values in the set of closed, densely defined, unbounded, positive operators in $X$, is a solution of the Riccati equation \myref{eq:evriccprima} if for all $t\ge t_0$ there is a set $D_R(t) \subset X$, dense in $[\ker Q_\infty]^\perp$, such that for every $x,y\in D_R(t)$ there exists $\<\frac{\ud}{\ud t}R(t)x,y\>_X$, all terms of \myref{eq:evriccprima} make sense and the equation holds.
\end{itemize}
\end{Definition}

\begin{Remark}\label{rm} {\bf (i)} {\em  
In the above definition the domain $D_P(t)$ varies with time, since its natural
choice (see the next theorem) is $D(A)\cap R(Q_t)$ which may change with time.
Similarly for the domain $D_R(t)$. Moreover $D_R(t)$ is assumed to be dense in $[\ker Q_\infty]^\perp$ and not in $X$, since its natural choice (see the next theorem) is $R(Q_t)$ which is indeed dense in $[\ker Q_\infty]^\perp$ and not in $X$, in general.\\[1mm]
{\bf (ii)} Note that we wrote equation \eqref{RiccatiHprima} without the initial condition: the reason is that we are interested to study all solutions of such equation, also in view of the study of the infinite horizon case, where the initial condition disappears. Clearly, looking at our original minimum energy problem (see Theorem \ref{th:th23Zab}-(iii)), the natural condition for \eqref{RiccatiHprima} (respectively \myref{eq:evriccprima}) is $P(0^+)=+\infty$ (respectively $R(0^+)=+\infty$); this condition, more precisely, reads as
$$\lim_{t\to 0^+} \langle P(t)x,x \rangle_H =+\infty \qquad \forall x\in \bigcap_{0<t\le \delta} D_P(t),$$
for seme $\delta>0$ (and similarly for $R(t)$). 
}
\hfill\qedo
\end{Remark}
We present now the following existence result.
\begin{Theorem}\label{PsoldiR} Suppose that Assumptions \ref{hp:main} and \ref{NC} hold. Then the operator $P_V(t)$ given by (\ref{formaP1}) is a solution of (\ref{RiccatiHprima}) on $[T_0,+\infty[\,$,
with the set $D_{P_V}(t)$ given by $D(A)\cap R(Q_t)$ for all $t\ge T_0$.

Moreover the operator $R_V(t)=Q_t^{-1}$
is a solution of \myref{eq:evriccprima} on $[T_0,+\infty[\,$,
with the set $D_{R_V}(t)$ given by $D(A)\cap R(Q_t)$ for all $t\ge T_0$.

\end{Theorem}
\begin{proof}


Fix $t\ge T_0$ and $x\in R(Q_t)$; then $P_V(t)x\in R(Q_\infty)$ since, using (\ref{formaP2}), $P_V(t)x = Q_\infty Q_t^{-1}x$ for all $x \in R(Q_t)$. Moreover from the definition of $Q_t$ and Assumption \ref{NC} it follows that $[\ker Q_t]^\perp$ is constant in $t$ for $t\ge T_0$, so that $Q_s^{-1}Q_s$ reduces to the identity on $[\ker Q_t]^\perp$ for $s$ and $t$ greater than $T_0$. Hence for $h \ne 0$ sufficiently small we can write for $x,y\in R(Q_t)$
\begin{eqnarray*}
\lefteqn{\left\langle \frac{P_V(t+h)-P_V(t)}{h}\,x,y \right\rangle_H = \left\langle \frac{P_V(t+h)- Q_\infty Q_{t+h}^{-1}Q_{t+h}Q_t^{-1}}{h}\,x,y \right\rangle_H }\\
& & = \left\langle \frac{P_V(t+h)[I-Q_{t+h}Q_t^{-1}]}{h}\,x,y \right\rangle_H = \left\langle P_V(t+h)\frac{[Q_t-Q_{t+h}]}{h}\,Q_t^{-1}x,y \right\rangle_H \\
& & = \left\langle Q_\infty^{-1/2}\frac{[Q_t-Q_{t+h}]}{h}\,Q_t^{-1}x,Q_\infty^{-1/2}P_V(t+h)y \right\rangle_X\,.
\end{eqnarray*}
Now we easily deduce, since $Q_\infty^{-1/2}e^{tA} \in {\cal L}(X)$ by Assumption \ref{NC},
\begin{eqnarray*}
Q_\infty^{-1/2}\frac{[Q_t-Q_{t+h}]}{h}\,Q_t^{-1}x & = & -\frac{1}{h} \int_t^{t+h}Q_\infty^{-1/2}e^{sA}BB^*e^{sA^*}Q_t^{-1}x \,\ud s  \\
& \to &  - Q_\infty^{-1/2}e^{tA}BB^*e^{tA^*}Q_t^{-1}x\quad \textrm{in } X \textrm{ as } h\to 0^+,
\end{eqnarray*}
so that, since $t\mapsto Q_\infty^{-1/2}P_V(t)y$ is continuous by Proposition \ref{valfun} (iii)(b), we readily obtain
\begin{eqnarray*}
\lim_{h\to 0} \left\langle \frac{P_V(t+h)-P_V(t)}{h}\,x,y \right\rangle_H  & = &
 - \langle Q_\infty^{-1/2}e^{tA}BB^*e^{tA^*} Q_t^{-1}x,Q_\infty^{-1/2}P_V(t)y\rangle_X \\
 & = & - \langle P_V(t)e^{tA}BB^*e^{tA^*} Q_t^{-1}x,y\rangle_H\,.
\end{eqnarray*}
This shows that
\begin{equation}\label{derP}
\exists \frac{\ud}{\ud t}\langle P_V(t)x,y\rangle_H = - \langle e^{tA}BB^*e^{tA^*} Q_t^{-1}x,P_V(t)y\rangle_H \qquad \forall x,y\in R(Q_t), \qquad \forall t \in [T_0,\infty[\,.
\end{equation}
Finally, using Proposition \ref{qtlyap}, for all $x,y\in R(Q_t)\cap D(A)$ we can compute for every $t\in [T_0, \infty[\,$:
\begin{eqnarray*}
\lefteqn{\frac{\ud}{\ud t}\langle P_V(t)x,y\rangle_H = - \langle Q_\infty^{-1/2}e^{tA}BB^*e^{tA^*} Q_t^{-1}x,Q_\infty^{-1/2}P_V(t)y\rangle_X} \\
& & = - \langle Q_\infty^{-1/2}e^{tA}BB^*e^{tA^*} Q_t^{-1}x,Q_\infty^{-1/2}[Q_\infty Q_t^{-1}]y\rangle_X \\[2mm]
& & = - \langle [AQ_t + (AQ_t)^* + BB^*] Q_t^{-1}x,Q_t^{-1}y\rangle_X \\[2mm]
& & = - \langle Ax,Q_t^{-1}y\rangle_X - \langle Q_t^{-1}x,Ay\rangle_X - \langle B^*Q_t^{-1}x,B^*Q_t^{-1}y\rangle_U \\[2mm]
& & = - \langle Ax,Q_\infty^{-1}P_V(t)y\rangle_X- \langle Q_\infty^{-1}P_V(t)x,Ay\rangle_X - \langle B^*Q_\infty^{-1}P_V(t)x,B^*Q_\infty^{-1} P_V(t)y\rangle_U \,.
\end{eqnarray*}
This completes the proof of the first statement. The proof of the second one
is completely similar and we omit it.
\end{proof}




\subsection{A partial uniqueness result}
\label{SSE:PARTIALUNIQ}

We are not able to prove a satisfactory uniqueness result; here is our statement
which establishes uniqueness in a restricted class
of solutions.

\begin{Theorem}\label{unicitaRicc} Suppose that Assumptions \ref{hp:main} and \ref{NC} hold. Let $P_V(t)$ be defined by \eqref{formaP1}. Let $S(t)$ be an operator defined in $[T_0,+\infty[\,$, with the following properties:
\begin{description}
\item[(i)]  $S(t)\in {\cal L}(H)$, $S(t)=S(t)^{*H}$, $\exists S(t)^{-1}\in {\cal L}(H)$
and the maps $t\mapsto S(t)$, $t\mapsto S(t)^{-1}$ are strongly continuous;
\item[(ii)] $S(t)^{-1}(Q_\infty(D(A^*)) \subseteq D(A)$ for every $t\in [T_0,\infty[\,$;
\item[(iii)]
for every $x \in S(t)^{-1}(Q_\infty(D(A^*)))$ and $t \ge T_0$ the map
$$
h \mapsto \frac1h(S(t+h)S(t)^{-1}-I)x
$$
is bounded in a neighborhood of $0$;
\item[(iv)]
for every $x,y\in S(t)^{-1}(R(Q_\infty))\cap D(A)$ the following equation holds:
$$\frac{\ud}{\ud t}\langle S(t)x,y \rangle_H = -\langle Ax, Q_\infty^{-1}S(t)y\rangle_X - \langle Q_\infty^{-1}S(t)x,Ay\rangle_X - \langle B^*Q_\infty^{-1}S(t)x,B^*Q_\infty^{-1}S(t)y \rangle_U\,;$$
\item[(v)] there exists $t_0\ge T_0$ such that $S(t_0)= P_V(t_0)$.
\end{description}
Then $S(t) \equiv P_V(t)$ in $[T_0,\infty[\,$.
\end{Theorem}
\begin{proof}
For fixed $t\ge T_0$, the above equation holds in particular for every $x,y\in S(t)^{-1}(Q_\infty(D(A^*)))$. Set now $\xi=S(t)x$, $\eta=S(t)y$: then we have $\xi,\eta\in Q_\infty(D(A^*))$ and, replacing $x$ and $y$ into (iii) above, we get
\begin{eqnarray*}
\lefteqn{\left[\frac{\ud}{\ud\tau}\langle S(\tau)S(t)^{-1}\xi,S(t)^{-1}\eta \rangle_H\right]_{\tau=t} }\\[1mm]
& & = -\langle AS(t)^{-1}\xi, Q_\infty^{-1}\eta\rangle_X - \langle Q_\infty^{-1}\xi,AS(t)^{-1}\eta\rangle_X - \langle B^*Q_\infty^{-1}\xi,B^*Q_\infty^{-1}\eta \rangle_U\,.
\end{eqnarray*}
Now we want to compute, whenever possible, $\frac{\ud}{\ud t}\< S(t)^{-1}\xi,\eta \>_H\,$.
We have
\begin{eqnarray*}
\<\frac{S(t+h)^{-1}- S(t)^{-1}}{h}\xi,\eta \>_H & = & \<S(t+h)^{-1}\left[\frac{S(t+h)- S(t)}{h}\right]S(t)^{-1}\xi,\eta \>_H \\[2mm]
&= & \<\left[\frac{S(t+h)- S(t)}{h}\right]S(t)^{-1}\xi,[S(t+h)^{-1}-S(t)^{-1}]\eta \>_H \\[2mm]
& & + \<\left[\frac{S(t+h)- S(t)}{h}\right]S(t)^{-1}\xi,S(t)^{-1}\eta \>_H.
%
%
%
\end{eqnarray*}
The second term clearly converges to
$$\left[\frac{\ud}{\ud\tau}\langle S(\tau)S(t)^{-1}\xi,S(t)^{-1}\eta \rangle_H\right]_{\tau=t}\,,$$ whereas the first term goes to $0$: indeed its first factor is bounded by assumption (iii), while the second one goes to $0$ in view of the strong continuity
of assumption (i). Hence we have
$$
\frac{\ud}{\ud t}\< S(t)^{-1}\xi,\eta \>_H
=
-\langle AS(t)^{-1}\xi, Q_\infty^{-1}\eta\rangle_X - \langle Q_\infty^{-1}\xi,AS(t)^{-1}\eta\rangle_X - \langle B^*Q_\infty^{-1}\xi,B^*Q_\infty^{-1}\eta \rangle_U\,.
$$
Now set $u_1=Q_\infty^{-1}\xi$, $v_1=Q_\infty^{-1}\eta$: by definition of pseudoinverses, we have $u_1,v_1\in [\ker Q_\infty]^\perp$ and, since  $\xi,\eta\in Q_\infty(D(A^*))$
there exist $u_0,v_0\in \ker Q_\infty$ such that $u:=u_1+u_0$, $v:=v_1+v_0$ belong to $D(A^*)$ and, of course, $Q_\infty u = \xi$ and $Q_\infty v = \eta$. The above equation then becomes
$$\frac{\ud}{\ud t}\langle S(t)^{-1}Q_\infty u,Q_\infty v\rangle_H =\langle AS(t)^{-1}Q_\infty u, v_1\rangle_X +\langle u_1,AS(t)^{-1}Q_\infty v\rangle_X +\langle B^*u_1,B^*v_1 \rangle_U \,.$$
Observe now that, by Proposition \ref{pr:Qt}-(ii), we have $B^*u_0=B^*v_0=0$. In addition, using assumption (ii) and the fact that $H\subseteq [\ker Q_\infty)]^\perp$, we get $S(t)^{-1}Q_\infty u =S(t)^{-1}\xi \in S(t)^{-1}(Q_\infty(D(A^*))) \subset D(A)\cap [\ker Q_\infty)]^\perp$, so that, by Lemma \ref{proreg}-(iv), $AS(t)^{-1}Q_\infty u \in [\ker Q_\infty)]^\perp$.
Hence we may write, for all $u,v\in D(A^*)$,
$$\frac{\ud}{\ud t}\langle S(t)^{-1}Q_\infty u,Q_\infty v\rangle_H =\langle AS(t)^{-1}Q_\infty u, v\rangle_X +\langle u,AS(t)^{-1}Q_\infty v\rangle_X +\langle B^*u,B^*v \rangle_U\,,$$
i.e.
$$\frac{\ud}{\ud t}\langle S(t)^{-1}Q_\infty u,v\rangle_X =\langle S(t)^{-1}Q_\infty u, A^*v\rangle_X +\langle A^*u,S(t)^{-1}Q_\infty v\rangle_X +\langle B^*u,B^*v \rangle_U, \ \ \forall u,v\in D(A^*).$$
This proves that $S(t)^{-1}Q_\infty$ solves the Lyapunov differential equation \myref{eq:Lyapdiff} in weak sense. Note that $S(t)^{-1}Q_\infty \in {\cal L}(X)$, since,
using also Lemma \ref{HH}-(iv),
\begin{eqnarray*}\|S(t)^{-1}Q_\infty x\|_X & = & \|Q^{1/2}_\infty S(t)^{-1}Q_\infty x\|_H\le \|Q^{1/2}_\infty\|_{{\cal L}(X)}\|S(t)^{-1}\|_{{\cal L}(H)}\|Q_\infty x\|_H \\
& \le & \|Q^{1/2}_\infty\|_{{\cal L}(X)}^2 \|S(t)^{-1}\|_{{\cal L}(H)} \|x\|_X\,,
\end{eqnarray*}
and it is selfadjoint, too, in view of
$$\langle S(t)^{-1}Q_\infty x,y\rangle_X =\langle S(t)^{-1}Q_\infty x,Q_\infty y\rangle_H = \langle Q_\infty x,S(t)^{-1}Q_\infty y\rangle_H = \langle x,S(t)^{-1}Q_\infty y\rangle_X \ \ \forall x,y\in X.$$
Now we recall that by \myref{formaP2} it follows that $P_V(t_0)^{-1} x = Q_{t_0}Q_\infty^{-1}x$ for every $x\in R(Q_\infty)$; then from the assumption $S(t_0)= P_V(t_0)$ we deduce
$$S(t_0)^{-1}Q_\infty z = P_V(t_0)^{-1}Q_\infty z= Q_{t_0}z \quad \forall z\in X.$$
Hence the operators $S(t)^{-1}Q_\infty$ and $Q_t$ solve the Lyapunov equation and coincide for $t=t_0$: thus they must coincide in $[T_0, \infty[$:
$$S(t)^{-1}Q_\infty = Q_t \qquad \forall t\ge T_0.$$
Thus for $x \in R(Q_\infty)$, i.e. $x= Q_\infty z$ with $z \in [\ker Q_\infty]^\perp$,
we may write
$$
S(t)^{-1}x = S(t)^{-1} Q_\infty z =
Q_t z=Q_t Q_\infty^{-1}x = P_V(t)^{-1}x.
$$
By density, we get $S(t)^{-1}x = P_V(t)^{-1}x$ for every $x\in H$, and finally $S(t)z\equiv P_V(t)z$ for every $z\in H$.
\end{proof}

\subsection{The selfadjoint commuting case}
\label{SSE:COMMUTING}

We consider now the case where $A$ is selfadjoint and commutes with $BB^*$. As a consequence, $A$ commutes with $Q_\infty$ and is selfadjoint in $H$, too. More specifically, from Proposition \ref{pr:Qt}-(v) we know that
$BB^*Q_\infty^{-1}=-2A$; hence in \myref{RiccatiHprima} the term
$$
-\< B^*Q_\infty^{-1}P(t)x,B^*Q_\infty^{-1}P(t)y \>_U=
-\< BB^*Q_\infty^{-1}P(t)x,Q_\infty^{-1}P(t)y \>_X
$$
can be simply rewritten as $2 \<AP(t)x,Q_\infty^{-1}P(t)y \>_X\,$; if in addition $AP(t)x\in H$, it just becomes $ \<AP(t)x,P(t)y \>_H\,$. Similarly, if $Ax, Ay \in H$, in \myref{RiccatiHprima} the terms $\< Ax, Q_\infty^{-1} P(t)y\>_X$ and $\< Q_\infty^{-1}P(t)x,Ay\>_X$ can be rewritten as $\< Ax, P(t)y\>_H$ and $\<P(t)x,Ay\>_H\,$.
Hence, in this case, we can rewrite \myref{RiccatiHprima} as
\begin{equation}\label{eq:commuting}
\frac{\ud}{\ud t}\< P(t)x,y \>_H
= -\< Ax, P(t)y\>_H - \<P(t)x,Ay\>_H + 2 \< AP(t)x,P(t)y \>_H.
\end{equation}
which makes sense for $x,y \in \overline D_P(t)$, where
$$
\overline D_P(t):=\left\{ z \in D(A): \ Az \in H,\ P(t)z \in D(A),\ AP(t)z \in H \right\}.
$$
We give now some statements about the solutions to this equation. The first one
(Theorem \ref{th:commuting1}) is an existence result under the null controllability assumption.
The subsequent ones (Theorems \ref{th:commuting2} and \ref{th:commuting3}) are uniqueness-type results and do not need null controllability.

\begin{Theorem}\label{th:commuting1}
Suppose that Assumptions \ref{hp:main} and \ref{NC} hold. Assume that $A$ is selfadjoint in $X$ and commutes with $BB^*$; let $K\in \Lc(H)$ be selfadjoint in $H$, non-negative, such that $AKA^{-1}\in \Lc(H)$. Let moreover $T_1\ge T_0$ be such that $(I-e^{tA}Ke^{tA})$ is invertible for each $t>T_1$.
Then $(I-e^{tA}Ke^{tA})^{-1}$ solves \eqref{eq:commuting} in $\,]T_1,\infty[\,$.
\end{Theorem}
\begin{proof}
It is clear that $T_1$ exists, since $e^{tA}$ is of negative type. Consider the set
$$D=\{z\in D(A)\cap H: \ Az\in H\};$$
it is dense in $H$, since it contains $Q_\infty(D(A))$, which is dense in $H$ by Lemma \ref{dens}: indeed, if $x\in Q_\infty(D(A))$, then $x=Q_\infty z$ with $z\in D(A)$, so that $Ax=AQ_\infty z= Q_\infty Az\in H$. Then, setting $\overline{P}(t)=(I-e^{tA}Ke^{tA})^{-1}$, we can write for $x,y\in D$ and $t>T_1$
\begin{eqnarray*}
\lefteqn{\frac{\ud}{\ud t} \langle \overline{P}(t)x,y\rangle_H = \frac{\ud}{\ud t} \langle (I-e^{tA}Ke^{tA})^{-1}x,y\rangle_H} \\
& & = -\langle (I-e^{tA}Ke^{tA})^{-1} (-Ae^{tA}Ke^{tA}-e^{tA}KAe^{tA})(I-e^{tA}Ke^{tA})^{-1} x,y\rangle_H \\[1mm]
& & = - \langle (I-e^{tA}Ke^{tA})^{-1} (-Ae^{tA}Ke^{tA}+A-e^{tA}Ke^{tA}A+A-2A)(I-e^{tA}Ke^{tA})^{-1} x,y\rangle_H \\[1mm]
& & = - \langle (I-e^{tA}Ke^{tA})^{-1} (-A(e^{tA}Ke^{tA}-I)-(e^{tA}Ke^{tA}-I)A-2A)(I-e^{tA}Ke^{tA})^{-1} x,y\rangle_H \\[1mm]
& & = - \langle (I-e^{tA}Ke^{tA})^{-1} Ax,y\rangle_H - \langle A(I-e^{tA}Ke^{tA})^{-1} x,y\rangle_H\\[1mm]
& & \qquad \qquad \qquad + \langle 2A(I-e^{tA}Ke^{tA})^{-1} x, (I-e^{tA}Ke^{tA})^{-1}y\rangle_H \\[1mm]
& & = -\langle Ax, \overline{P}(t)y\rangle_H -\langle \overline{P}(t)x,Ay\rangle_H+\langle 2A\overline{P}(t)x,\overline{P}(t)y\rangle_H.
\end{eqnarray*}
This shows that $\overline{P}$ solves \eqref{eq:commuting} with $\overline{D}_{\overline{P}}(t)=D$ for every $t>T_1$. Note that
$$\overline{P}(t) = (I-e^{(t-T_1)A}Le^{(t-T_1)A})^{-1}, \qquad L=e^{T_1A}Ke^{T_1A}=I-\overline{P}(T_1)^{-1}.$$
\end{proof}
The second statement is a uniqueness result.
\begin{Theorem}\label{th:commuting2}
Suppose that Assumption \ref{hp:main} holds. Assume that $A$ is selfadjoint in $X$ and commutes with $BB^*$. Let moreover $S:[T^*,\infty[\,\to \Lc(H)$ be a strongly continuous, selfadjoint, non-negative operator, such that: \\[1mm]
{\bf (i)} $S(t)$ is invertible for $t\ge T^*$ and $S(T^*)=S^*$;\\[1mm]
{\bf (ii)} $S(t)$ solves \eqref{eq:commuting} in $\,]T^*,\infty[\,$.\\[1mm]
Then $S(t)=(I-e^{(t-T^*)A}Le^{(t-T^*)A})^{-1}$ for every $t\ge T^*$, where $L=I-(S^*)^{-1}$.
\end{Theorem}
\begin{proof} Set again $D=\{z\in D(A)\cap H: \ Az\in H\}$ and define
$$U(t):= S(t)^{-1}, \qquad t\ge T^*.$$
Obviously, $U(T^*)=(S^*)^{-1}$. Moreover, for every $t>T^*$ and $x,y$ in the set $S(t)(\overline{D}_S(t))$, which is dense in $H$, we have by (ii)
\begin{eqnarray*}
\frac{\ud}{\ud t} \langle U(t)x,y\rangle_H & = &  -\left[\frac{\ud}{\ud t} \langle S(t)\xi,\eta\rangle_H\right]_{\xi=U(t)x,\ \eta=U(t)y} \\[1mm]
& = & \langle A\xi,S(t)\eta\rangle_H +\langle S(t)\xi,A\eta\rangle_H- 2\langle AS(t)\xi,S(t)\eta\rangle_H \\[1mm]
& = & \langle AU(t)x,y\rangle_H +\langle x,AU(t)y\rangle_H- 2\langle Ax,y\rangle_H \\[1mm]
& = & \langle U(t)x,Ay\rangle_H +\langle Ax,U(t)y\rangle_H- 2\langle Ax,y\rangle_H \,.
\end{eqnarray*}
This is a linear equation, governed by the semigroup $P \mapsto e^{tA}Pe^{tA}$: by the variation of constants formula we have for each $x,y\in S(t)(\overline{D}_S(t))$
\begin{eqnarray*}
U(t)x & = & e^{(t-T^*)A}U(T^*)e^{(t-T^*)A}x - 2\int_{T^*}^t e^{(t-s)A}Ae^{(t-s)A}x\,ds \\[2mm]
& = & e^{(t-T^*)A}(1-L)e^{(t-T^*)A}x +\int_{T^*}^t \frac{\ud}{\ud s} e^{2(t-s)A}x \,ds \\[2mm]
& = & e^{(t-T^*)A}(1-L)e^{(t-T^*)A}x -e^{2(t-T^*)A}x+x = (I-e^{(t-T^*)A}Le^{(t-T^*)A})x.
\end{eqnarray*}
By density this shows that
$$S(t)^{-1} = U(t)=I-e^{(t-T^*)A}Le^{(t-T^*)A},$$
which is our claim.
\end{proof}

In the next result we look at non-invertible solutions obtained through projections.

\begin{Theorem}\label{th:commuting3} Suppose that Assumption \ref{hp:main} holds. Assume that
$A$ is selfadjoint in $X$ and commutes with $BB^*$; let $S:[T^*,\infty[\,\to \Lc(H)$ be a strongly continuous, selfadjoint, non-negative operator, which solves \eqref{eq:commuting} in $\,]T^*,\infty[\,$. Let moreover $P\in {\cal L}(H)$ be an orthogonal projection, such that $AP=PA$ and $S(t)P(D(A)\cap H) \subseteq D(A)$ for every $t>T^*$. \\
Then $PS(t)P$ solves \eqref{eq:commuting} in $\,]T^*,\infty[\,$ if and only if $S(t)P(D(A)\cap H)\subseteq R(P)$ for every $t>T^*$.
\end{Theorem}
\begin{proof}
We start by observing that the existence of a projection $P$ in $H$ such that $AP=PA$ implies that $A$ maps $D(A)\cap H$ into $H$: indeed if $z\in D(A)\cap H$ we have $Az = APz+A(I-P)z = PAz + (I-P)Az$ and both the terms of the last member belong to $H$. \\
Suppose that
\begin{equation}\label{eq:vincoli}
AP=PA, \qquad S(t)P(D(A)) \subseteq D(A)\cap R(P).
\end{equation}
As $S(t)$ solves \eqref{eq:commuting}, we have for $x,y\in \overline{D}_S(t)$
\begin{equation}\label{eq:DREcomm}
\frac{\ud}{\ud t}\langle S(t)x,y \rangle_H
= -\langle Ax, S(t)y\rangle_H - \langle S(t)x,Ay\rangle_H + 2 \langle AS(t)x,S(t)y \rangle_H\,,
\end{equation}
where, as we know,
$$\overline{D}_S(t):=\{ z \in D(A)\cap H: \ Az \in H,\ S(t)z \in D(A),\ AS(t)z \in H \}.$$
Now, if $z\in D$ (i.e $z\in D(A)\cap H$ and $Az\in H$) then, by \eqref{eq:vincoli}, $Pz\in D(A)$ with $APz=PAz\in H$, and in addition $S(t)Pz\in R(P)\cap D(A)$, so that $AS(t)Pz=APS(t)Pz=PAS(t)Pz\in H$. Thus $Pz\in D_S(t)$ for each $t>T^*$ and $z\in D$. Hence, setting
$$\overline{D}_{PSP}(t) = D \qquad \forall t>T^*,$$
and replacing in \eqref{eq:DREcomm} $x,y$ by $Px,Py$, we have for every $x,y\in \overline{D}_{PSP}(t)$ and $t>T^*$
$$\frac{\ud}{\ud t}\langle S(t)Px,Py \rangle_H
= -\langle APx, S(t)Py\rangle_H - \langle S(t)Px,APy\rangle_H + 2 \langle AS(t)Px,S(t)Py \rangle_H\,,
$$
i.e.
$$\frac{\ud}{\ud t}\langle PS(t)Px,y \rangle_H
= -\langle Ax, PS(t)Py\rangle_H - \langle PS(t)Px,Ay\rangle_H + 2 \langle AS(t)Px,S(t)Py \rangle_H\,.
$$
Now we remark that $S(t)Px=PS(t)Px$ and $S(t)Py=PS(t)Py$; hence we obtain, for every $x,y\in \overline{D}_{PSP}(t)$ and $t>T^*$,
$$\frac{\ud}{\ud t}\langle PS(t)Px,y \rangle_H
= -\langle Ax, PS(t)Py\rangle_H - \langle PS(t)Px,Ay\rangle_H + 2 \langle APS(t)Px,PS(t)Py \rangle_H\,.
$$
This shows that $PS(t)P$ solves \eqref{eq:commuting} in $\,]T^*,\infty[\,$.\\[2mm]
Suppose conversely that $P\in {\cal L}(H)$ is an orthogonal projection, such that $AP=PA$, $S(t)P(D(A)\cap H)\subseteq D(A)$ for every $t>T^*$ and $PS(t)P$ solves \eqref{eq:commuting} in $\,]T^*,\infty[\,$. Assume by contradiction that for some $t>T^*$ there exists $v\in S(t)P(D(A)\cap H)\setminus R(P)$: we can write $v=S(t)Pz$ with $z\in D(A)\cap H$. Then $w=(I-P)S(t)Pz$ belongs to $D(A)\cap R(P)^\perp$, $w\neq 0$ and $Aw=(I-P)AS(t)Pz \in R(P)^\perp$. As $\overline{D}_{PSP}(t)$ is dense in $H$, there exists $\{z_n\}\subset \overline{D}_{PSP}(t)$ such that $z_n \to z$ in $H$; then $w_n=(I-P)S(t)Pz_n \to w$ in $H$ and consequently $w_n\neq 0$ for sufficiently large $n$.\\
Now by assumption we have for every $x,y\in \overline{D}_{PSP}(t)$
$$\frac{\ud}{\ud t}\langle PS(t)Px,y \rangle_H +\langle Ax, S(t)y\rangle_H + \langle PS(t)Px,Ay\rangle_H - 2 \langle APS(t)Px,PS(t)Py \rangle_H = 0,$$
whereas for every $x,y\in \overline{D}_S(t)$ it holds
$$\frac{\ud}{\ud t}\langle S(t)x,y \rangle_H +\langle Ax, S(t)y\rangle_H + \langle S(t)x,Ay\rangle_H - 2 \langle APS(t)Px,PS(t)Py \rangle_H = 0.$$
We may choose $x=y=z_n$ in the first equation and $x=y=Pz_n$ in the second one: indeed, as $z_n\in \overline{D}_{PSP}(t)$, we have $Pz_n\in D(A)\cap H$ and $APz_n=PAz_n\in H$; hence $S(t)Pz_n\in D(A)$ and consequently, as remarked at the beginning of the proof, $AS(t)Pz_n\in H$: this shows that  $Pz_n\in \overline{D}_S(t)$. Thus we get
$$\frac{\ud}{\ud t}\langle PS(t)Pz_n,z_n \rangle_H +\langle Az_n, PS(t)Pz_n\rangle_H + \langle PS(t)Pz_n,Az_n\rangle_H - 2 \langle APS(t)Pz_n,PS(t)Pz_n \rangle_H = 0$$
and
$$\frac{\ud}{\ud t}\langle S(t)Pz_n,Pz_n \rangle_H +\langle APz_n, S(t)Pz_n\rangle_H + \langle S(t)Pz_n,APz_n\rangle_H - 2 \langle AS(t)P^2z_n,S(t)P^2z_n \rangle_H = 0.$$
The second equation can be rewritten as
$$\frac{\ud}{\ud t}\langle PS(t)Pz_n,z_n \rangle_H+\langle Az_n, PS(t)Pz_n\rangle_H + \langle PS(t)Pz_n,Az_n\rangle_H - 2 \langle AS(t)Pz_n,S(t)Pz_n \rangle_H = 0.$$
Subtracting the second equation from the first one, we get
$$\langle AS(t)Pz_n,S(t)Pz_n \rangle_H - \langle APS(t)Pz_n,PS(t)Pz_n \rangle_H = 0.$$
On the other hand
\begin{eqnarray*}
\lefteqn{0=\langle AS(t)Pz_n,S(t)Pz_n \rangle_H - \langle APS(t)Pz_n,PS(t)Pz_n \rangle_H} \\
& & = \langle A(I-P)S(t)Pz_n,S(t)Pz_n \rangle_H + \langle APS(t)Pz_n,(I-P)S(t)Pz_n \rangle_H \\
& & = \langle A(I-P)S(t)Pz_n,(I-P)S(t)Pz_n \rangle_H +\langle PAS(t)Pz_n,(I-P)S(t)Pz_n \rangle_H \\
& & = \langle Aw_n,w_n \rangle_H + 0 = \langle Aw_n,w_n \rangle_H.
\end{eqnarray*}
Now we recall that $A$ is of negative type and selfadjoint in $H$: thus, since $w_n\neq 0$,
$$\langle Aw_n,w_n \rangle_H = - \langle (-A)w_n,w_n \rangle_H = - \|(-A)^{1/2}w_n\|_H^2 <0:$$
this is a contradiction.
\end{proof}

\section{Examples}\label{examples}

\subsection{Delay state equation}

Consider the following linear controlled delay equation
\begin{equation}
\left\{
\begin{array}
[c]{l}
x'(t)  =a_0x(t)+ a_1x(t-d)+ b_0 u(t)
,\text{ \ \ \ }t\in[ 0,T] \\
x(0)=x_0, \quad x(s)=x_1(s), \; s \in [-d,0[\,,
\end{array}
\right.   \label{eq-contr-rit}
\end{equation}
where the initial datum $(x_0,x_1)$ is in $\R\times L^2(-d,0;\R)$, the control $u$ belongs to $L^2(0,T;\R)$ and the coefficients $a_0$, $a_1$, $b_0$ are real numbers with $a_1\ne 0$ and $b_0\ne 0$ to avoid degeneracy. We call $x(\cdot\,;(x_0,x_1),u)$ the unique solution which always exists (see e.g. \cite[Chapter 4]{BDDM07}).
Using a standard approach (see e.g. again \cite[Chapter 4]{BDDM07}),
we reformulate equation (\ref{eq-contr-rit}) as an abstract differential equation in the Hilbert space $\calh=\R\times L^2(-d,0;\R)$.
To this end we introduce the operator
$A : \cald(A) \subset \calh
\rightarrow \calh$ as follows:
\begin{equation}\label{A}
\left\{
\begin{array}{l}
\cald(A)=\left\lbrace (x_0,x_1)\in \calh:\ y_1\in W^{1,2}([-d,0],\R),\ x_1(0)=x_0 \right\rbrace,
\\[2mm]
A(x_0 ,x_1 )= ( a_0 x_0+ a_1x_1(-d), x_1').
\end{array}
\right.
\end{equation}
We denote by $e^{tA}$ the $C_0$-semigroup generated by $A$: for
$x=(x_0,x_1)\in \calh$,
\begin{equation}
\label{semigroup}
e^{tA} \left(x_0, x_1\right)=
\left(x(t;(x_0,x_1),0),x(t+\cdot;(x_0,x_1),0)\right)\in \calh.
\end{equation}
The control operator $B$ is bounded and defined as
\begin{equation}
 \label{B}
B:\R\rightarrow \calh,\qquad Bu=(b_0 u,0), \quad u\in\R.
\end{equation}
In this setup, equation \eqref{eq-contr-rit} is equivalent (in the sense that the first component of $y$ is the solution of \eqref{eq-contr-rit}) to
the equation in $\calh$:
$$
y'(t)=Ay(t)+Bu(t),\quad y(0)=(x_0,x_1) \in \calh.
$$
For this system the null controllability Assumption \ref{NC} holds for any $T_0>r$, see e.g. \cite[Theorem 10.2.3]{DaPratoZabczyk96} or \cite{OlbrotPandolfi88}.
Hence Theorems \ref{PsoldiR} and \ref{unicitaRicc} hold in this case.

Now we compute the adjoints and the controllability operator.
We denote by $A^*$ the adjoint operator of $A$:
\begin{equation} \label{Astar}
\left\{
\begin{array}{l}
\cald(A^*)=\left\lbrace (x_0,x_1)\in \calh:\ y_1\in W^{1,2}([-d,0],\R),\ x_1(-d)=a_1x_0 \right\rbrace,
\\[2mm]
A^*(x_0 ,x_1 )=( a_0 x_0 +x_1(0), -x_1').
\end{array}
\right.
\end{equation}
Similarly, denoting by $e^{tA^*}=(e^{tA})^*$ the $C_0$-semigroup generated by $A^*$,
we have for $(x_0,x_1)\in \calh$
\begin{equation}
\label{semigroupadjoint}
e^{tA^*} (x_0,x_1)=
\left(x(t;(y_0,y_1),0),x(t+\cdot;(y_0,y_1),0)\right) \in \calh
\end{equation}
where
\begin{equation}
 \label{eq:foradjoint}
y_0=x_0, \quad and \quad  y_1(r)=a_1^{-1}x_1(-d-r), \quad r \in \,]-d,0].
\end{equation}
The adjoint of the control operator is
\begin{equation}
 \label{B*}
B^*:\calh \rightarrow \R,\qquad B^*(x_0,x_1)=b_0 x_0, \quad
\forall (x_0,x_1)\in\calh.
\end{equation}
It follows that
$$
BB^*e^{tA^*}(x_0,x_1)= b_0^2 \left( x(t;(y_0,y_1),0),0 \right)
$$
where $(y_0,y_1)$ is as in \eqref{eq:foradjoint}.
Hence, by linearity of \eqref{eq-contr-rit} we can write
$$
e^{tA}BB^*e^{tA^*}(x_0,x_1)= b_0^2 x(t;(y_0,y_1),0) \left( g(t),g(t+\cdot) \right)
$$
where again $(y_0,y_1)$ is as in \eqref{eq:foradjoint} and
$g(t)=x(t;(1,0),0)$ (which is a given piecewise polynomial function that may be computed recursively).
We can then finally write, for $(x_0,x_1)\in\calh$,
\begin{equation}
\label{eq:Qtdelay}
Q_t (x_0,x_1)= b_0^2 \left(\int_0^t x(s;(y_0,y_1),0) g(s)ds,
\int_0^t x(s;(y_0,y_1),0) g(s+\cdot)ds\right) \in \calh
\end{equation}
where $(y_0,y_1)$ is as in \eqref{eq:foradjoint}.
It is not obvious to compute $R(Q_t)$ and $R(Q_t^{1/2})$.
However we can at least say that $R(Q_t) \subseteq D(A)$: indeed the boundary condition $x_0=x_1(0)$ is obviously satisfied for all elements of $R(Q_t)$ by
continuity of translations in $L^2$; on the other hand the second element of
$Q_t(x_0,x_1)$ belongs to $W^{1,2}([-d,0],\R)$ by direct verification
simply using the continuity of $x(s;(y_0,y_1),0)$.

Hence the sets
$D_P(t)$ and $D_R(t)$ in Theorem \ref{PsoldiR} are equal to $R(Q_t)$
in this case.

\subsection{Diagonal cases}
\label{SSE:DIAGONAL}
Let $\{e_n\}_{n\in \N}$ be a complete orthonormal
system in the Hilbert space $X$, and let $\{\lambda_n\}_{n\in
\N}$ be a strictly increasing sequence of strictly positive numbers such that
$\lambda_n \to +\infty$ as $n\to \infty$. We define on the space $X$
the semigroup
$$S(t)=\sum_{n\in \N} e^{-\lambda_nt}\langle x,e_n\rangle_X\,
e_n\, ,\quad t\ge 0.$$
It is easily verified that $S$ is an analytic semigroup of negative type $-\omega$, where
$\omega=\min_{n\in \N} \lambda_n=\lambda_0>0$, with norm $\|S_(t)\|_{{\cal
L}(X)} = e^{-\omega t}$. Its generator is the self-adjoint, dissipative, densely defined operator $A:D(A)\subset X\to X$, given by
\begin{equation}\label{defA}
\left\{ \begin{array}{l}
D(A) = \left\{ x\in X: \sum_{n\in \N} \lambda_n^2 \langle x,e_n \rangle_X^2 <+\infty \right \}\\[2mm]
Ax= -\sum_{n\in \N} \lambda_ n\langle x,e_n \rangle_X \, e_n
\end{array} \right.
\end{equation}
(see \cite[pp. 178 and 198]{Zabczyk92}). Note that $0\in \rho(A)$ and that $A^{-1}$ is selfadjoint and compact.\\
As $A$ is dissipative, the fractional powers $(-A)^\alpha$ of
$-A$ are well defined (see \cite[Proposition 6.1, page 113]{BDDM07}). \\
Concerning the operator $B$, we assume that $B:U\to X$ is such that $BB^*$ is diagonal in $X$:
$$BB^*e_n = b_ne_n \quad \forall n\in \N,$$
with $b_n\ge 0$ for all $n\in \N$. By Assumption \ref{hp:main} $B$ is bounded, hence the sequence $\{b_n\}$ must be bounded, too. However here we generalize a bit the setting, allowing $BB^*$ to be unbounded. Since $S(t)=e^{tA}$ commutes with $BB^*$ we have, see \eqref{eq:Qtcommutingformula},
$$Q_t x = \int_0^t e^{2sA}BB^*x \,ds = \frac12 A^{-1}(e^{2tA}-I)BB^*x, \quad \forall t>0, \qquad Q_\infty x = -\frac12 A^{-1}BB^*x;$$
in particular, for $t>0$,
\begin{equation}\label{Qdiag}
Q_t e_n = \frac1{2\lambda_n}(1-e^{-2\lambda_nt}) b_n e_n, \quad Q_\infty e_n = \frac1{2\lambda_n} b_n e_n \qquad \forall n\in \N.
\end{equation}
Thus, if $BB^*$ is possibly unbounded, we need to assume
\begin{equation}\label{bnlambdan}
\sup_{n\in \N} \frac{b_n}{\lambda_n} <\infty
\end{equation}
in order that $Q_t,Q_\infty \in {\cal L}(X)$ for all $t>0$.
The null controllability holds for a given $t>0$ if and only if there exists $c_t>0$ such that
$$\|S(t)x\|_X^2\le c_t \langle Q_t x,x\rangle_X \qquad \forall x\in X.$$
This is equivalent to
$$e^{-2\lambda_n t} \le c_t \frac{b_n}{2\lambda_n}(1-e^{-2\lambda_nt}) \qquad \forall n\in \N.$$
Hence Assumption \ref{NC} holds for every $T_0>0$ if and only if $b_n>0$ for every $n\in \N$ and
$$\sup_{n \in \N}
\frac{2\lambda_n}{b_n(e^{2\lambda_n t}-1)}< + \infty \quad \forall t>0.$$
Now, we look at $R(Q_\infty)$ and $R(Q_\infty^{1/2})$ (observe that, by Proposition \ref{monotQt}-(iii), these are equal to $R(Q_t)$ and $R(Q_t^{1/2})$ for all $t>0$).
By \eqref{Qdiag} is clear that $R(Q_\infty)\subseteq R(BB^*)$ and $R(Q_\infty^{1/2})\subseteq \overline{R(BB^*)}$.
\begin{itemize}
  \item If $b_n\ne 0$ only for a finite number of $n\in \N$ then, clearly,
$R(Q_\infty)=R(Q_\infty^{1/2})=R(BB^*)\subseteq D(A)$. In this case the RE is substantially finite dimensional: the function $t\to Q_t^{-1}$ is a solution on $D_P(t)=D_R(t)=R(BB^*)$ and, by Theorem \ref{th:commuting3}, $P Q_t^{-1}P$ is a solution for every projection generated by some elements of the basis $\{e_n\}$.

  \item If $b_n\ne 0$ for every $n\in \N_1$, where $\N_1$ is an infinite subset of $\N$, then, clearly,
$$R(Q_\infty) = \left\{z\in R(BB^*): \
\left\{\frac{\lambda_n}{b_n}\langle z,e_n\rangle_X\right\}_{n\in \N_1} \in \ell^2\right\}.$$
In this case the RE is infinite dimensional. Again the function $t\to Q_t^{-1}$ is a solution on $D_P(t)=D_R(t)=D(A)\cap R(Q_\infty)$ and, by Theorem \ref{th:commuting3}, $P Q_t^{-1}P$ is a solution for every projection generated by some elements of the basis $\{e_n\}$.
\end{itemize}

We now look closely at the second case above, when $\N_1=\N$. First, if $BB^*$ is bounded, i.e.
$b=\{b_n\}_{n\in \N}\in \ell^\infty$, we have $R(Q_\infty)\subseteq D(A)$ and, similarly
$R(Q_\infty^{1/2})\subseteq D(A^{1/2})$. On the other hand, if, for some $\delta >0$, we have $b_n \ge \delta$ for all $n\in \N$, then
$R(Q_\infty)\supseteq D(A)$ and $R(Q_\infty^{1/2})\supseteq D(A^{1/2})$.

Thus, if both $BB^*$ and $(BB^*)^{-1}$ are bounded we have $R(Q_\infty)= D(A)$ and $R(Q_\infty^{1/2})=D(A^{1/2})$.

Finally, if $b_n=\lambda_n^\alpha$ for every $n\in \N$, with  $\alpha \in \R$, then $R(Q_\infty)=D(A^{1-\alpha})$ and $R(Q_\infty^{1/2})=D(A^{\frac{1-\alpha}2})$.\\

Now we consider a special case which fits into the application studied e.g. in \cite{BDGJL4}
in the case of the Landau-Ginzburg model.
We take $X=H^{-1}(0,\pi;\R)$ and $A$ the Laplacian in $X$ with Dirichlet boundary conditions.
We also take $U=X$ and $B=I$. Using what said just above we see that $R(Q_\infty)=D(A)=H^{1}_0(0,\pi;\R)$ and $H=R(Q_\infty^{1/2})=D(A^{1/2})=L^{2}(0,\pi;\R)$.


\section*{Appendix}
\appendix
\section{Pseudoinverses}
\label{SS:PSEUDOINVERSES}
We recall here two well known results of functional analysis that will be very useful in the sequel.

Given a linear operator $F: X \to Y$, where $X$ and $Y$ are Hilbert spaces,
we define, as in \cite[p. 209]{Zabczyk92} (see also \cite[p. 429]{DaPratoZabczyk14}), the pseudoinverse $F^{-1}$ of $F$ as the linear operator
$$
F^{-1}: D(F^{-1})\subset Y \to X,
$$
with domain
$D(F^{-1})=R(F)$,
where $F^{-1}y$ is the element of $F^{-1}(\{y\})$ with minimal norm.
Note that $R(F^{-1})= (\ker F)^\perp$.

We have the following result, taken from
\cite[Proposition B.1, p.429]{DaPratoZabczyk14}.
\begin{Proposition}
\label{Pr:B.1}
Let $E,E_1,E_2$ be three Hilbert spaces, let $A_1: E_1 \to E$,
$A_2:E_2 \to E$ be linear operators, let $A_1^*:E \to E_1$ and
$A_2^*:E \to E_2$ be their adjoints and finally let
$A_1^{-1}:D(A_1^{-1})\subseteq E_1 \to E$,
$A_2^{-1}:D(A_2^{-1})\subseteq E_2 \to E$
be the respective pseudoinverses. Then we have:
\begin{itemize}
  \item[(i)]  $R(A_{1})\subseteq R(A_{2})$ if and only if there exists a constant $k>0$
such that
$$\|A_{1}^{ \star }x\|_{E_1}\leq k\|A_{2}^{ \star }x\|_{E_2} \qquad \forall x
\in E.$$
  \item[(ii)] If
$$\|A_{1}^{ \star }x\|_{E_1}=\|A_{2}^{\star}x\|_{E_2} \qquad \forall x \in E,$$
then $R(A_{1})$=$R(A_{2})$, $R(A_{1}^{-1})=R(A_{2}^{-1})$ and
$$\|A_{1}^{-1}x\|_{E_1}=\|A_{2}^{-2}x\|_{E_2}\qquad \forall x \in  R(A_{1}).$$
\end{itemize}
\end{Proposition}

\section{Some properties of commuting operators}\label{SSE:COMMOPER}

Given a real separable Hilbert space $X$,
let $A:D(A)\subseteq X \to X$ be a generator of a strongly continuous semigroup $e^{tA}$
and, for any $\lambda \in \rho(A)$, denote by $R(\lambda,A)$ the resolvent operator $(\lambda - A)^{-1}$.

\begin{Definition}
Consider an operator $K \in \Lc(X)$. We say that $K$ commutes with $A$ if, for all
$x \in D(A)$ we have $Kx \in D(A)$ and $AKx=KAx$.
In particular this means that $K$ maps $D(A)$ into itself \footnote{In this context the operator $AK$ may be defined on a set $Y$ strictly larger than $D(A)$ and in this case, in addition, the operator $KA$ can be extended to all of $Y$. An obvious example of this situation occurs when $K$ is a resolvent of $A$.}.
\end{Definition}

The following result is known but, for the reader's convenience, we provide the complete proof as we could not find it in the literature.
\begin{Lemma}
\label{lm:commop}
Let $X$ be a Hilbert space, $K\in \Lc(X)$, and $A:D(A)\subseteq X \to X$ be the generator of a strongly continuous semigroup $e^{tA}$.
The following statements are equivalent:
\begin{itemize}
  \item[(i)] There exists $\lambda_0 \in \rho (A)$ such that
  $R(\lambda_0,A)K=KR(\lambda_0,A)$.
  \item[(ii)]  For every $\lambda \in \rho (A)$ it holds
  $R(\lambda,A)K=KR(\lambda,A)$.
  \item[(iii)] $K$ commutes with $A$.
  \item[(iv)] $K^*$ commutes with $A^*$.
  \item[(v)] For all $t>0$ we have $e^{tA}K=Ke^{tA}$.
\end{itemize}
\end{Lemma}

\begin{proof}
(i) $\Longleftrightarrow$ (ii).

We only prove (i) $\Longrightarrow$ (ii), as the other direction is obvious.
Let $\lambda \in \rho (A)$. Then
$$
KR(\lambda,A) - R(\lambda,A)K=K[R(\lambda,A)-R(\lambda_0,A)]
+ [KR(\lambda_0,A)-R(\lambda_0,A)K]+[R(\lambda_0,A)-R(\lambda,A)]K,
$$
and so, using the so-called resolvent identity, and the fact that $K$ commutes with $R(\lambda_0,A)$,
\begin{eqnarray*}
KR(\lambda,A) - R(\lambda,A)K & = & K(\lambda_0- \lambda)R(\lambda,A)R(\lambda_0,A)
-(\lambda_0- \lambda)R(\lambda,A)R(\lambda_0,A)K =\\
& = & (\lambda_0- \lambda)R(\lambda_0,A)[KR(\lambda,A) - R(\lambda,A)K].
\end{eqnarray*}
Then it follows that
$$
[I - (\lambda_0- \lambda)R(\lambda_0,A)][KR(\lambda,A) - R(\lambda,A)K]=0.
$$
Since $I - (\lambda_0- \lambda)R(\lambda_0,A)=(\lambda - A)R(\lambda_0,A)$
we can apply $R(\lambda,A)$ to both sides of the above equality which, thanks to the injectivity of $R(\lambda,A)$ is equivalent to
$$
R(\lambda_0,A)[KR(\lambda,A) - R(\lambda,A)K]=0;
$$
this, using the injectivity of $R(\lambda_0,A)$, gives the claim.\\[2mm]
(ii) $\Longrightarrow$ (iii).

For sufficiently large $n \in \N$, consider $A_n:= nAR(n,A)=-n + n^2 R(n,A)$, the Yosida approximants of $A$.
By (iii) we immediately have $KA_n x=A_nKx$ for all $x \in X$.
Let now $x \in D(A)$. By the properties of Yosida approximants \cite{Pazy83} we have
$A_nx \to Ax$ and $x_n:=nR(n,A)Kx \to Kx$ as $n \to + \infty$.

Now $Ax_n=A_nKx=KA_nx\to KAx$ as $n \to + \infty$. Since $A$ is closed, we have $Kx \in D(A)$ and $AKx=KAx$, which is the claim.\\[2mm]
(iii) $\Longrightarrow$ (ii).

Let $\lambda \in \rho(A)$. We have, for $x \in X$,
$$
(\lambda -A) [R(\lambda,A)K-KR(\lambda,A)]x
=Kx - \lambda K R(\lambda,A)x + AK R(\lambda,A)x
$$
and, since $A$ commutes with $K$,
$$
=Kx - K\lambda  R(\lambda,A)x + KA R(\lambda,A)x
=K[I-\lambda  R(\lambda,A)+A R(\lambda,A)]x=0.
$$
By the injectivity of $\lambda-A$ this implies that
$R(\lambda,A)Kx-KR(\lambda,A)x=0$.\\[2mm]
(iii) $\Longleftrightarrow$ (iv).

We only prove (iii) $\Longrightarrow$ (iv), as the other direction follows simply by taking the adjoints and using the relations $A^{**}=A$ and $K^{**}=K$.

Let $x\in D(A)$ and $y \in D(A^*)$. Then
$$
\<KAx,y\>=\<Ax,K^*y\>
$$
and also, since $K$ commutes with $A$,
$$
\<KAx,y\>=\<AKx,y\>=\<Kx,A^*y\>=\<x,K^*A^*y\>.
$$
From the two above it follows that
$$
|\<Ax,K^*y\>|=|\<x,K^*A^*y\>| \le C |x|
$$
for some $C>0$. This means that $K^*y\in D(A^*)$ and that
$$
A^*K^*y=(KA)^*y=(AK)^*y =K^*A^*y
$$
which is the claim.\\[2mm]
(ii) $\Longrightarrow$ (v).

We know that, for all $x \in X$, $A_nK x=KA_n x$ and $e^{tA_n}x=\sum_{p=0}^{\infty}\frac{t^pA_n^p}{p!}x$; hence we get
$$
Ke^{tA_n}x=K\sum_{p=0}^{\infty}\frac{t^pA_n^p }{p!}x
=\sum_{p=0}^{\infty}K\frac{t^pA_n^p }{p!}x
=\sum_{p=0}^{\infty}\frac{t^pA_n^p }{p!}Kx
=e^{tA_n}Kx.
$$
We now let $n \to \infty$ and use the fact that, by the properties of Yosida approximants
\cite{Pazy83}, $e^{tA_n}x \to e^{tA}x$ for all $x \in X$, as $n \to + \infty$. This implies
the claim.\\[2mm]
(v) $\Longrightarrow$ (ii).

We know \cite{Pazy83} that, for all sufficiently large $\lambda \in \rho (A)$ and for all $x \in X$,
$$
R(\lambda, A)x=\int_0^{\infty}e^{-\lambda t} e^{tA}x\, \ud t.
$$
Then
$$
KR(\lambda, A)x=K\int_0^{\infty}e^{-\lambda t} e^{tA}x\, \ud t
=
\int_0^{\infty}K e^{-\lambda t} e^{tA}x\, \ud t
=
\int_0^{\infty} e^{-\lambda t} e^{tA}Kx\, \ud t
=
R(\lambda, A)Kx
$$
and the claim follows.
\end{proof}

We will need also the following result on pseudoinverses.

\begin{Lemma}
\label{lm:commoppseudoinv}
Let $E$ be a Hilbert space and let $A_1, A_2 \in \Lc(E)$ be such that $A_1A_2=A_2A_1$, $A_2$
is selfadjoint and $R(A_1)\subseteq {R(A_2)}$.
Then, denoting by $A_2^{-1}$ the pseudoinverse of $A_2\,$, the two operators
$$
A_1 A_2^{-1}:R(A_2) \to E \qquad \textrm{and} \qquad
A_2^{-1}A_1 :E \to E
$$
coincide on $R(A_2)$; hence, in particular, $A_1 A_2^{-1}$ can be extended to all of $E$.
\end{Lemma}

\begin{proof}
Take $z \in R(A_2)$ and set
$$
v:=A_1 A_2^{-1}z, \qquad u:=A_2^{-1}A_1z.
$$
Applying $A_2$ we get
$$
A_2v=A_2A_1 A_2^{-1}z=A_1 A_2A_2^{-1}z=A_1 z,
\qquad
A_2u=A_2 A_2^{-1}A_1z=A_1 z
$$
where in the first equality we have used the commuting assumption. This means that $A_2(u-v)=0$, i.e. $u-v\in \ker A_2\,$. Now by the definition of pseudoinverse we have $u \in (\ker A_2)^\perp$,
while $v\in R(A_1)\subseteq R(A_2) \subseteq (\ker A_2)^\perp$, since $A_2$ is selfadjoint. Hence
it must be $u-v=0$ and the result follows.
\end{proof}

\section{Controllability operators and minimum energy}
\label{SSE:CONTROLLABOPERATORS}

Following \cite[p. 209]{Zabczyk92}, we collect some basic properties of the controllability operators $Q_t$ defined in \eqref{qt}:

\begin{Proposition}\label{pr:Qt} Let $Q_t$ be defined by \eqref{qt}.
\begin{itemize}
  \item[(i)] The operator
$Q_t$ is linear, bounded, selfadjoint and non-negative.
  \item[(ii)] For $0\le s \le t \le +\infty$ it holds
\begin{equation}\label{eq:kerQtinclusions}
\ker Q_t \subseteq \ker Q_s \subseteq \ker B^*=\ker BB^*
\end{equation}
and each inclusion becomes an equality when $BB^*$ and $A$ commute.
  \item[(iii)] For $0\le s \le t \le +\infty$,
\begin{equation}\label{eq:kerQtperpinclusions}
[\ker Q_t]^\perp \supseteq [\ker Q_s]^\perp \supseteq [\ker B^*]^\perp=
[\ker BB^*]^\perp,
\end{equation}
so that
\begin{equation}\label{eq:ImQtinclusions}
\overline{R(Q_t)} \supseteq \overline{R(Q_s)} \supseteq \overline{R(B)}=\overline{R(BB^*)},
\end{equation}
and each inclusion becomes an equality when $BB^*$ and $A$ commute.
  \item[(iv)] For $0\le t \le \tau \le +\infty$ we have
  \begin{equation}\label{eq:Qttauprima}
Q_\tau= Q_t + e^{tA}Q_{\tau-t}e^{tA^*}.
\end{equation}
  \item[(v)]
Finally, if $A$ is selfadjoint, and $A$ and $BB^*$ commute, we have for all $x \in X$
  \begin{equation}\label{eq:Qtcommutingformula}
Q_t x= \frac12 A^{-1}\left(e^{2tA}- I\right)BB^* x, \quad t>0;
\qquad\qquad Q_\infty x= -\frac12 A^{-1}BB^* x.
\end{equation}
This, in particular, implies that for every $t\in [0,+\infty]$ the operator $Q_t$ commutes with $A$ and that
  \begin{equation}\label{eq:Qtcommutingformulabis}
2Ay =-BB^*Q_\infty^{-1}y \qquad \forall y\in R(Q_\infty) \subseteq D(A).
\end{equation}
\end{itemize}
\end{Proposition}
\begin{proof}
The statement (i) is immediate by definition of $Q_t$.

We prove now (ii). Indeed, for every $t \in [0,+\infty]$, since $Q_t$ is selfadjoint we have
$$\begin{array}{rrcll} Q_t x=0 & \Longleftrightarrow & \langle Q_t x,x \rangle_X =0 & \Longleftrightarrow & \displaystyle \int_0^t \|B^*e^{rA^*}x\|_U^2\,\ud r=0 \\[4mm]
& & & \Longleftrightarrow & \|B^*e^{rA^*}x\|_U=0 \quad \textrm{for a.e. } r \in [0,t]. \end{array}$$
The above immediately gives $\ker Q_t \subseteq \ker Q_s $ when $s \le t$. Moreover,
since $r \to \|B^*e^{rA^*}x\|_U$ is continuous, this function is identically $0$,
so the last implies $\|B^*x\|_U=0$. Finally since $BB^*$ is selfadjoint then
$BB^*x=0$ is equivalent to $B^*x=0$.

If $A$ and $BB^*$ commute then, by Lemma \ref{lm:commop} also $A^*$ and $BB^*$ commute and so also $e^{tA^*}$ and $BB^*$.
It follows that, if $BB^*x=0$ then, for all $t>0$,
$$
\<Q_t x,x\>_X=\int_0^t\<BB^*e^{rA^*}x,e^{rA^*}x\>_X\,\ud r
=\int_0^t\<e^{rA^*}BB^*x,e^{rA^*}x\>_X\,\ud r=0,
$$
which gives the claim.

Concerning (iii), (\ref{eq:kerQtperpinclusions}), as well as
(\ref{eq:ImQtinclusions}), easily follow from (\ref{eq:kerQtinclusions}).

The statement (iv) follows by writing
$$
Q_\tau x=\int_0^\tau  e^{rA}BB^*e^{rA^*}\, \ud r=
\int_0^t e^{rA}BB^*e^{rA^*}dr + \int_t^\tau e^{rA}BB^*e^{rA^*}\, \ud r
$$
and then changing variable in the second integral.

The statement (v) follows since in this case, by Lemma \ref{lm:commop},
$e^{tA}$ commutes with $BB^*$. Hence
$$
Q_t x=\int_0^t  e^{rA}BB^*e^{rA}x\, \ud r=
\int_0^t  e^{2rA}BB^*x \,\ud r
$$
and \myref{eq:Qtcommutingformula} follows by standard integration of semigroups.
Concerning the commutativity of $Q_t$ and $A$ we first observe that, by \myref{eq:Qtcommutingformula} we have $R(Q_t) \subseteq D(A)$ for $t\in\, ]0,\infty]$.
Moreover, still by \myref{eq:Qtcommutingformula}, we have, by direct computations
$$
2AQ_tx=2Q_tA x =(e^{2tA}-I)BB^*x, \qquad
2AQ_\infty x=2Q_\infty A x =-BB^*x
$$
for all $x \in D(A)$.
Finally, for any given $y\in R(Q_\infty)$ we set
$x:=Q_\infty^{-1}y\in [\ker Q_\infty]^\perp \subseteq X$ and we write, using the last formula
$$
2AQ_\infty Q_\infty^{-1}y =-BB^*Q_\infty^{-1}y,
$$
which, by the properties of the pseudoinverses, gives \myref{eq:Qtcommutingformulabis}.
\end{proof}


Finally we provide the following, partly well known result, concerning the images of the controllability operators.

\begin{Proposition}\label{monotQt}
Assume that Assumption \ref{hp:main} holds.
\begin{itemize}
  \item[(i)]If $0<t<\tau < \infty$ then $R(Q_t^{1/2}) \subseteq R(Q_\tau^{1/2})\subseteq R(Q_\infty^{1/2})$.
  \item[(ii)]If, in addition, the system (\ref{eq:state-fin-new}) is null-controllable at time $T_0$, i.e. Assumption \ref{NC} holds, then $R(Q_t^{1/2}) = R(Q_\infty^{1/2})$ for all $t\in [T_0,\infty[$.
  \item[(iii)] If $A$ is selfadjoint and commutes with $BB^*$ then, without assuming null controllability,
  for all $t\in\, ]0,\infty[\,$, the equalities $R(Q_t) = R(Q_\infty)$ and
$R(Q_t^{1/2}) = R(Q_\infty^{1/2})$ hold.
\end{itemize}
\end{Proposition}
\begin{proof} The results (i) and (ii) concerning the images of
the operators $Q_t^{1/2}$, $t\in[0,\infty]$ are well known: see e.g., for point (i)
the proof of Theorem 2.2 in Part IV, Chapter 2 of \cite{Zabczyk92}; for point (ii) the proof of Theorem 2.2 in
\cite{DaPratoPritchardZabczyk91}.

We now prove (iii). For all $0<t<\tau\le \infty$ we have, using \myref{eq:Qttauprima},
the selfadjointness of $A$ and the commutativity,
\begin{equation}\label{eq:Qttau}
Q_\tau= Q_t + e^{2tA}Q_{\tau-t}=Q_t +Q_{\tau-t} e^{2tA};
\end{equation}
hence, if $\tau=\infty$ we get, for all $x \in H$ and $t\ge 0$,
\begin{equation}\label{eq:Qtinfty}
Q_t x=Q_\infty x -Q_{\infty} e^{2tA}x=Q_\infty(x-  e^{2tA}x).
\end{equation}
Thus we immediately get $R(Q_t)\subseteq R(Q_\infty)$ for all $t\ge 0$.
On the other hand we have, for $x \in H$ and $t\ge 0$,
$$
Q_\infty x=Q_t x - e^{2tA}Q_{\infty} x,
$$
which implies, for all $x \in H$ and $t\ge 0$,
$$
\|Q_\infty x \|_X\le \|Q_t x\|_X + Me^{-2\omega t}\|Q_{\infty}x\|_X\,.
$$
Let $T_1\ge 0$ be such that $Me^{-2\omega T_1}=1$. Then for all $t>T_1$ the above implies
$$
\|Q_\infty x \|_X\le \frac{1}{1-Me^{-2\omega t}}\|Q_t x\|_X\,.
$$
Using Proposition \ref{Pr:B.1}-(i) this implies that $R(Q_t)=R(Q_\infty)$ for all $t>T_1$.
If $T_1=0$ the claim is proved. If $T_1>0$ take $t\le T_1$. We have, taking $\tau=2t$
in \myref{eq:Qttau},
\begin{equation}\label{eq:Qtt}
Q_{2t}x= Q_t x+ Q_{t} e^{2tA}x=Q_t (x+ e^{2tA}x), \quad x \in H.
\end{equation}
This implies that $R(Q_{2t})\subseteq R(Q_t)$. Iterating this argument we see
that it must be $R(Q_{2^k t})\subseteq R(Q_t)$ for all $k \in \N$.
Taking $\bar k$ such that $2^{\bar k} t>T_1$ we then get
$R(Q_\infty)=R(Q_{2^{\bar k}  t})\subseteq R(Q_t)$. This proves the claim.

Concerning the last statement we observe that, by \myref{eq:Qtinfty}
and since $e^{tA}$ commutes with $Q^{1/2}_{\infty}$, too, we may write
\begin{eqnarray*}
\|Q^{1/2}_\infty x \|_X^2 & = & \<Q_\infty x, x\>_X= \<Q_t x,x\>_X - \<e^{2tA}Q_{\infty} x,x\>_X \\
& = & \<Q_t x,x\>_X -  \<e^{tA}Q^{1/2}_{\infty} x,e^{tA}Q^{1/2}_{\infty} x\>_X \\
& = & \|Q^{1/2}_t x \|_X^2 +\|e^{tA}Q^{1/2}_{\infty} x\|_X^2 \le \|Q^{1/2}_t x \|_X^2 +Me^{-\omega t}\|Q^{1/2}_{\infty} x\|_X^2.
\end{eqnarray*}
Hence, taking $T_2$ such that $Me^{-\omega T_2}=1$ (i.e. $T_2=2T_1$), for $t>T_2$
we get
$$
\|Q^{1/2}_\infty x \|_X^2 \le \frac{1}{1-Me^{-\omega t}}\|Q^{1/2}_t x \|_X^2
$$
which gives $R(Q^{1/2}_\infty )\subseteq R(Q^{1/2}_t )$, and hence $R(Q^{1/2}_\infty )= R(Q^{1/2}_t )$, for $t>T_2$. If $T_2=0$ the claim follows.
Otherwise, using \myref{eq:Qtt}, we have
\begin{eqnarray*}
\|Q^{1/2}_{2t} x \|_X^2 & = & \<Q_{2t} x, x\>_X= \<Q_t x,x\>_X + \<e^{2tA}Q_{t} x,x\>_X \\
& = & \<Q_t x,x\>_X +  \<e^{tA}Q^{1/2}_{t} x,e^{tA}Q^{1/2}_{t} x\>_X \\
& = & \|Q^{1/2}_t x \|_X^2 +\|e^{tA}Q^{1/2}_{t} x\|_X^2 \le \|Q^{1/2}_t x \|_X^2 +Me^{-\omega t}\|Q^{1/2}_{t} x\|_X^2.
\end{eqnarray*}
Hence, arguing as above we get $R(Q^{1/2}_\infty )=R(Q^{1/2}_t )$ for $t>0$.
\end{proof}

We have the following result about the optimal pairs when $x \in R(Q_t)$.
\begin{Proposition}\label{pr:optcouplet}
Let $x \in R(Q_t)$. Let $(\hat y_{t,x},\hat u_{t,x})$ be the optimal pair in $[-t,0]$. Then we have
\begin{equation}\label{eq:optcontrt}
\hat u_{t,x}(r) = B^* e^{-rA^*} { Q}^{-1}_{t} x
\quad \forall r\in [-t,0],
\end{equation}
with ${ Q}^{-1}_t$ defined as in Theorem \ref{th:th23Zab} (iii). Moreover the corresponding optimal state $\hat y$ satisfies
\begin{equation}\label{eq:optstatet}
\hat y_{t,x}(r)=Q_{t+r}e^{-rA^*}Q_t^{-1}x, \qquad r \in [-t,0];
\end{equation}
hence the optimal pair satisfies the feedback formula
\begin{equation}\label{eq:optfeedbackt}
\hat u_{t,x}(r)=B^* Q_{t+r}^{-1}\hat y_{t,x}(r), \qquad r \in \,]-t,0],
\end{equation}
and, formally, $\hat y_{t,x}$ is a solution of the backward closed loop equation (BCLE)
\begin{equation}\label{eq:CLEt}
y'(r)=(A+BB^* Q_{t+r}^{-1}) y(r) , \qquad r \in\, ]-t,0]
\end{equation}
with final condition $y(0)=x$.
\end{Proposition}

\begin{proof}
Formula (\ref{eq:optcontrt}) follows from \cite[Theorem 2.3-(iii), page 210]{Zabczyk92}.

Formula (\ref{eq:optstatet}) follows by inserting (\ref{eq:optcontrt}) into
the state equation:
\begin{eqnarray*}
\hat y_{t,x}(r) & = & \int_{-t}^r e^{(r-s)A}B \hat u_{t,x}(s)\,\ud s =
 \int_{-t}^r e^{(r-s)A}B B^* e^{-sA^*} { Q}^{-1}_{t} x \,\ud s \\
& = & \left( \int_{-t}^r e^{(r-s)A}B B^* e^{(r-s)A^*}\, \ud s \right)e^{-rA^*} { Q}^{-1}_{t} x =Q_{t+r}e^{-rA^*} Q^{-1}_{t} x;
\end{eqnarray*}
moreover formula (\ref{eq:optfeedbackt}) follows by simply observing that $\hat y_{t,x}(r) \in R(Q_{t+r})$ for every $r \in [-t,0]$, and using (\ref{eq:optcontrt})-(\ref{eq:optstatet}).\\
Inserting (\ref{eq:optfeedbackt}) into the state equation (\ref{eq:state-fin-new}) we see that, formally, $\hat y_{t,x}$ is a solution of the BCLE (\ref{eq:CLEt}).\\
\end{proof}

\begin{Remark}
\label{rm:optcoupleslyap}
{\rm
It is not hard to show that the above result holds true also in the case when $t=+\infty$. So we have for the optimal pair the representations
\begin{equation}\label{eq:optcontrtinfty}
\hat u_{\infty,x}(r) = B^* e^{-rA^*} { Q}^{-1}_{\infty} x
\qquad  r\in\, ]-\infty,0],
\end{equation}
\begin{equation}\label{eq:optstatetinfty}
\hat y_{\infty,x}(r)=Q_{\infty}e^{-rA^*}Q_\infty^{-1}x, \qquad r \in \,]-\infty,0];
\end{equation}
and the feedback formula
\begin{equation}\label{eq:optfeedbacktinfty}
\hat u_{\infty,x}(r)=B^* Q_{\infty}^{-1}\hat y_{\infty,x}(r), \qquad r \in \,]-\infty,0].
\end{equation}
Thus, formally, $\hat y_{\infty,x}$ is a solution of the backward closed loop equation (BCLE)
\begin{equation}\label{eq:CLEtinfty}
y'(r)=(A+BB^* Q_{\infty}^{-1}) y(r) , \qquad r \in \,]-\infty,0]
\end{equation}
with final condition $y(0)=x$.
Using the Lyapunov equation \myref{eq:LyapunovX} proved in Proposition \ref{qtlyap}, the above (\ref{eq:CLEtinfty}) can be simplified as
\begin{equation}\label{eq:CLEtinftybis}
y'(r)=-Q_{\infty}A^*Q_{\infty}^{-1} y(r) , \qquad r \in \,]-\infty,0].
\end{equation}
Hence, if $A^*$ commutes with $Q_\infty$ (e.g. when $A$ is selfadjoint,
and $A$ and $BB^*$ commute), then the BCLE \myref{eq:CLEtinfty}
becomes
$$
y'(r)=-A^* y(r) , \qquad r \in \,]-\infty,0].
$$
which is well posed and is solved by the optimal trajectory.
The same argument can be used in the finite horizon case of Proposition \ref{pr:optcouplet}
to rewrite \myref{eq:CLEt} but, due to the presence of $Q_t^\prime$ in the Lyapunov equation \myref{eq:Lyapdiff}, the result is not so useful.} \hfill\qedo
\end{Remark}
We now give a counterexample\footnote{We are indebted to Giorgio Fabbri for this example.} in the case where the null controllability Assumption \ref{NC} does not hold.
\begin{Example}
{\em
Let us consider the Hilbert spaces $X=L^2(0,1)$ and $U= \mathbb{R}$. The operator
$$\left\{ \begin{array}{l} D(A)= \{ f\in H^1(0,1): f(0)=0 \} \\[1mm]
Af=-f' \end{array} \right. $$
is the infinitesimal generator in $L^2(0,1)$ of the $C_0$-semigroup
(see e.g. \cite[Chapter I, Section 4.c]{EngelNagelbook}, or
\cite{BarucciGozzi01}):
$$(e^{tA}f)(s) := \left\{ \begin{array}{lll}
f(s-t) & \text{if }& s > t \\[1mm]
0 & \text{if }& s\le t,\end{array}\right.\qquad s\in\, ]0,1[\,, \quad t\ge 0,$$
or, in other words,
$$(e^{tA}f)(\cdot)= f(\cdot-t)\, \chi_{[t,1]}(\cdot).$$
Next, let $B:\mathbb{R} \to L^2(0,1)$ be defined by
$$B(a) := a \, \chi_{[0,1/4]}(\cdot).$$
Consider the state equation
$$\left\{ \begin{array}{l} y'(s) = A y(s) + B u(s)\\
y(0)=0. \end{array}\right.$$
For any fixed $t\in \,]0,1]$, we have $f\in Q_t^{1/2}(L^2(0,1))$ if and only if there exists $u\in L^2(0,t)$ such that
$$f = \int_0^t e^{(t-r)A}Bu(r)\, \ud r.$$
By the definition of $B$ and the explicit form of $e^{tA}$ we easily get
$$f(\cdot) = \int_0^t u(r) \, \chi_{[0,1/4]}(\cdot - t+r) \chi_{[t-r,1]}(\cdot)\,\ud r = \int_0^t u(r) \, \chi_{[t-r,1\wedge (t-r+1/4)]}(\cdot)\, \ud r.$$
Now fix $t=1/4$: then if $f\in Q_{1/4}^{1/2}(L^2(0,1))$ it holds
$$f(s) = \int_0^{1/4} u(r) \, \chi_{[1/4-r,1/2-r]}(s)\, \ud r,$$
so that necessarily $f(s)=0$ for all $s\in \,]1/2,1]$.
On the other hand, take $t=1$ and $u\equiv 1\in L^2(0,1)$; then if $f = \int_0^1 e^{(1-r)A}Bu(r)\, \ud r$ we have in particular $f\in Q_1^{1/2}(L^2(0,1))$ and
$$f(s) = \int_0^1 \chi_{[1-r,1\wedge(5/4-r)]}(s)\,\ud r = \int_0^1 \chi_{[1-s,1\wedge(5/4-s)]}(r)\,\ud r = s\wedge \frac14  \quad \forall s\in [0,1].$$
This shows that $f$ cannot belong to $Q_{1/4}^{1/2}(L^2(0,1))$, i.e.
$$Q_1^{1/2}(L^2(0,1)) \nsubseteq Q_{1/4}^{1/2}(L^2(0,1)),$$
and in particular, the system cannot be null controllable at any $T\in \,]0,1/4]$.
}
\hfill\qedo
\end{Example}

\section{Proofs}\label{SSE:PROOFS}

{\bf Proof of Lemma \ref{HH}}.\\[1mm]
We start proving (i). Let $\{x_n\}$ be a Cauchy sequence in $H$: then for each $n$ we have $x_n = Q_\infty^{1/2}z_n$, where $z_n\in [\ker Q_\infty^{1/2}]^\perp$ is uniquely determined, and by (\ref{nH}) $\{Q_\infty^{-1/2}x_n\}=\{z_n\}$ is a Cauchy sequence in $X$, so that it converges to some $z\in [\ker Q_\infty^{1/2}]^\perp$. As $Q_\infty^{1/2}\in {\cal L}(X)$, $\{x_n\}$ is a Cauchy sequence in $X$, too, and it converges to some $x\in X$. It follows that $Q_\infty^{1/2}z = x$, so that $x\in H$ and $x_n \to x$ in $H$. This shows that $H$ is complete. To prove that $H$ is continuously embedded into $X$, take $x\in H$: then $x=Q_\infty^{1/2}y$ for a unique $y= Q_\infty^{-1/2}x \in [\ker Q_\infty^{1/2}]^\perp$. Thus
$$\|x\|_X = \|Q_\infty^{1/2}y\|_X \le \|Q_\infty^{1/2}\|_{{\cal L}(X)}\|y\|_X = \|Q_\infty^{1/2}\|_{{\cal L}(X)}\|Q_\infty^{-1/2}x\|_X = \|Q_\infty^{1/2}\|_{{\cal L}(X)}\|x\|_H\,.$$

Concerning (ii), let $x \in H$. Then there exists a unique $z \in [\ker Q_\infty^{1/2}]^\perp =\overline{R(Q_\infty^{1/2})} $ such that $x=Q_\infty^{1/2}z$.
So there exists a sequence $\{z_n\}\subset R(Q_\infty^{1/2})$ such that
$z_n \to z$ in $X$ as $n \to + \infty$. Setting $x_n=Q_\infty^{1/2}z_n$ we have as $n\to \infty$
$$
\|x_n-x\|_H =\|Q_\infty^{1/2}(z_n-z)\|_H =\|z_n-z\|_X \to 0,
$$
and the claim follows.\\[1mm]
The statement (iii) follows from (\ref{nH}) by just taking $x=y$.\\[0.5mm]
To prove the statement (iv) we observe first that, for all $x \in H$ with $x=Q^{1/2}_\infty z$,
$z \in X$,
$$
\frac{\|Q^{1/2}_\infty x\|_H}{\|x\|_H}=
\frac{\|x\|_X}{\|z\|_X}\frac{\|Q^{1/2}_\infty z\|_X}{\|z\|_X},
$$
which implies $\|Q^{1/2}_\infty \|_{\Lc(H)}\le \|Q^{1/2}_\infty \|_{\Lc(X)}$.
On the other hand, if $z_n\in X$ is such that
$$
\frac{\|Q^{1/2}_\infty z_n\|_X}{\|z_n\|_X}>\|Q^{1/2}_\infty \|_{\Lc(X)}-\frac1n
$$
then, setting $x_n=Q^{1/2}_\infty z_n \in H$, we have
$$
\|Q^{1/2}_\infty \|_{\Lc(H)}\ge \frac{\|Q^{1/2}_\infty x_n\|_H}{\|x_n\|_H}
=
\frac{\|Q^{1/2}_\infty z_n\|_X}{\|z_n\|_X}>
\|Q^{1/2}_\infty \|_{\Lc(X)}-\frac1n.
$$
which gives the claim.\\[1mm]
Finally, (v) follows by observing that
$Q_\infty^{-1/2}F$ is a well defined closed linear operator from $X$ to $X$ and applying the closed graph theorem.\hfill \qed \\[2mm]
{\bf Proof of Lemma \ref{dens}}\\[1mm]
We just consider the case $t=\infty$, since the case $T_0\le t <\infty$ is quite similar. Fix $x\in H$. Then there is a unique $z\in [\ker Q_\infty]^\perp$ such that $Q_\infty^{1/2}z=x$. As $[\ker Q_\infty]^\perp = \overline{R(Q_\infty)}=\overline{R(Q_\infty^{1/2})}$, there exists $\{z_n\}\subset X$ such that $Q_\infty^{1/2}z_n \to z$ in $X$. Since $D(A^*)$ is dense in $X$, for each $n\in \N^+$ we can find $y_n \in D(A^*)$ such that $\|y_n-z_n\|_X<1/n$, so that $Q_\infty^{1/2}y_n \to z$ in $X$, too. Hence
$$\|Q_\infty y_n - x\|_H = \|Q_\infty^{1/2} y_n - Q_\infty^{-1/2}x\|_X = \|Q_\infty^{1/2} y_n - z\|_X\to 0,$$
i.e. $x$ belongs to the closure of $Q_\infty(D(A^*))$ in $H$.
The density of $D(A) \cap H$ follows since, by Lemma \ref{proreg}-(i), we have
$Q_\infty(D(A^*))\subset D(A)\cap H$. Finally, the last statement follows by
Lemma \ref{proreg}-(iv).\hfill \qed \\[2mm]
{\bf Proof of Lemma \ref{relinH}}\\[1mm]
We recall that for $0<t\le \infty$ we have
\begin{equation}\label{pseu}
Q_t^{1/2}Q_t^{-1/2}z = z \quad \forall z\in R(Q_t^{1/2}), \qquad Q_t^{-1/2}Q_t^{1/2}x = P_{[\ker Q_t]^\perp}\,x \quad \forall x\in X.
\end{equation}
By (\ref{pseu}), using also the fact that $Q_t^{1/2}$ is selfadjoint in $X$, we get (i):
\begin{eqnarray*}
\langle z, Q_t^{-1/2}w\rangle_X & = & \langle Q_t^{1/2}Q_t^{-1/2}z, Q_t^{-1/2}w\rangle_X \\
& = & \langle Q_t^{-1/2}z, Q_t^{1/2}Q_t^{-1/2}w\rangle_X = \langle Q_t^{-1/2}z, w\rangle_X\qquad \forall z,w\in R(Q_t^{1/2}).
\end{eqnarray*}
About (ii), we have by (\ref{pseu}), (\ref{nH}) and (i):
\begin{eqnarray*}
\langle Q_\infty^{1/2}x,y\rangle_H & = & \langle Q_\infty^{1/2}x,Q_\infty^{1/2}Q_\infty^{-1/2}y\rangle_H = \langle x, Q_\infty^{-1/2}y\rangle_X = \langle Q_\infty^{-1/2}x, y\rangle_X \\[1mm]
& = & \langle Q_\infty^{-1/2}x, P_{[\ker Q_\infty]^\perp}\,y \rangle_X = \langle Q_\infty^{-1/2}x, Q_\infty^{-1/2}Q_\infty^{1/2}y\rangle_X \\[0.5mm]
& = & \langle x, Q_\infty^{1/2}y\rangle_H \qquad \forall x,y\in H.
\end{eqnarray*}
Finally, (iii) follows immediately by applying (ii) twice.\hfill \qed \\[2mm]
{\bf Proof of Lemma \ref{composQt}}\\[1mm]
For fixed $t\ge T_0$, let $x\in H$ be such that $Q_t^{1/2}Q_\infty^{-1/2}x=0$: then, by (\ref{pseu}),
$$Q_\infty^{-1/2}x = Q_t^{-1/2}(Q_t^{1/2}Q_\infty^{-1/2}x) = P_{[\ker Q_t]^\perp}Q_\infty^{-1/2}x =0,$$
which implies $Q_\infty^{-1/2}x \in \ker Q_t= \ker Q_\infty$. On the other hand, by definition we also have $Q_\infty^{-1/2}x \in [\ker Q_\infty]^\perp$, so that $Q_\infty^{-1/2}x=0$ and consequently $x=Q_\infty^{1/2}(Q_\infty^{-1/2}x)=0$. This proves that $Q_t^{1/2}Q_\infty^{-1/2}$ is one-to-one. Moreover, for each $y\in H$ the equation $Q_t^{1/2}Q_\infty^{-1/2}x=y$ is equivalent to $P_{[\ker Q_t]^\perp}Q_\infty^{-1/2}x = Q_t^{-1/2}y$; but since $Q_\infty^{-1/2}x\in [\ker Q_\infty]^\perp =[\ker Q_t]^\perp$, we deduce $Q_\infty^{-1/2}x = Q_t^{-1/2}y$ and hence $x=Q_\infty^{1/2}Q_t^{-1/2}y$. This shows that
$Q_t^{1/2}Q_\infty^{-1/2}$ is surjective.\\
We now claim that $Q_t^{1/2}Q_\infty^{-1/2}$ has closed graph in $H\times H$. Indeed, let $\{x_n\}$ be a sequence in $H$ such that $(x_n,Q_t^{1/2}Q_\infty^{-1/2}x_n)\to (x,y)$ in $H\times H$. This means, by definition,
\begin{equation}\label{conv1}
Q_\infty^{-1/2}x_n \to Q_\infty^{-1/2}x \quad \textrm{in } X,
\end{equation}
and
\begin{equation}\label{conv2}
Q_\infty^{-1/2}Q_t^{1/2}Q_\infty^{-1/2}x_n \to Q_\infty^{-1/2}y \quad \textrm{in } X;
\end{equation}
if we apply $Q_t^{1/2}$ to (\ref{conv1}) we obtain
$$Q_t^{1/2}Q_\infty^{-1/2}x_n \to Q_t^{1/2}Q_\infty^{-1/2}x \quad \textrm{in } X,$$
whereas if we apply $Q_\infty^{1/2}$ to (\ref{conv2}) we get
$$Q_t^{1/2}Q_\infty^{-1/2}x_n \to y \quad \textrm{in } X,$$
so that $y=Q_t^{1/2}Q_\infty^{-1/2}x$ and our claim is proved. By Lemma \ref{HH} it follows that $Q_t^{1/2}Q_\infty^{-1/2}\in {\cal L}(H)$.
Finally, the inverse $Q_\infty^{1/2}Q_t^{-1/2}$ is also in ${\cal L}(H)$ by the same argument, or by the open mapping theorem.

The proof of the second statement is quite analogous.\hfill \qed\\[2mm]
{\bf Proof of Lemma \ref{composQt2}}\\[1mm]
It is clear that $Q_t^{-1/2}Q_\infty^{1/2}$ maps $X$ into $X$ and vanishes on $\ker Q_\infty=\ker Q_\infty^{1/2}$. Moreover if $y\in [\ker Q_\infty]^\perp$ then $x = Q_\infty^{-1/2}Q_t^{1/2} y$ is in $[\ker Q_\infty]^\perp$ and satisfies $Q_t^{-1/2}Q_\infty^{1/2}x=y$, so that $Q_t^{-1/2}Q_\infty^{1/2}$ is one-to-one from $[\ker Q_\infty]^\perp$ onto itself. \\
Now we prove that $Q_t^{-1/2}Q_\infty^{1/2}$ has closed graph. Let $(x_n,Q_t^{-1/2}Q_\infty^{1/2}x_n) \to (x,y)$ in $X\times X$: then $x_n \to x$ in $X$ and $Q_t^{-1/2}Q_\infty^{1/2}x_n \to y$ in $X$, so that $y\in [\ker Q_\infty]^\perp$. It follows that
$$Q_\infty^{1/2}x_n \to Q_\infty^{1/2}x$$
and also
$$Q_\infty^{1/2}x_n = Q_t^{1/2}[Q_t^{-1/2}Q_\infty^{1/2}x_n]\to Q_t^{1/2}y,$$
so that $Q_\infty^{1/2}x=Q_t^{1/2}y$; but since $y\in [\ker Q_\infty]^\perp$, we deduce $y=Q_t^{-1/2}Q_\infty^{1/2}x$. Thus $Q_t^{-1/2}Q_\infty^{1/2}\in {\cal L}(X)$ and the result follows.
\hfill \qed\\[2mm]
{\bf Proof of Lemma \ref{aggH}}\\[1mm]
Indeed, by (\ref{nH}) and Lemma \ref{relinH} (ii)-(iii),
\begin{eqnarray*}
\langle Lx,y\rangle_H & = & \langle Q_\infty^{-1/2}Lx,Q_\infty ^{-1/2}y\rangle_X = \langle Lx,Q_\infty ^{-1}y\rangle_X \\[1mm]
& = & \langle x,L^*Q_\infty ^{-1}y\rangle_X = \langle Q_\infty^{1/2}x,Q_\infty^{1/2}L^*Q_\infty ^{-1}y\rangle_H = \langle x,Q_\infty L^*Q_\infty^{-1}y\rangle_H\,.
\end{eqnarray*}
\hfill\qed\\[2mm]
{\bf Proof of Proposition \ref{valfun}}\\[1mm]
(i) Fix $x\in H$ and $t, \tau>0$ with $\tau<t$. For any $u\in {\cal U}_{[0,\tau]}(x)$, define
$$\tilde{u}(s)=\left\{ \begin{array}{ll}0 & \textrm{if } \ s\in [0,t-\tau] \\[1mm] u(s-t+\tau) & \textrm{if } \ s\in\, ]t-\tau,t]. \end{array}\right.$$
We have $\tilde{u} \in {\cal U}_{[0,t]}(x)$, since obviously
$$\int_0^t e^{(t-s)A}B\tilde{u}(s)\,\ud s = \int_{t-\tau}^t e^{(t-s)A}Bu(s-t+\tau)\,\ud s = \int_0^\tau e^{(\tau-\sigma)A}Bu(\sigma)\,\ud \sigma =x,$$
and moreover
$$\int_0^t \|\tilde{u}(s)\|_U^2 \, \ud s = \int_{t-\tau}^t \|u(s-t+\tau)\|_U^2 \, \ud s = \int_0^\tau \|u(\sigma)\|_U^2 \, \ud \sigma.$$
Now, for a fixed $\varepsilon>0$ we may select $u\in {\cal U}_{[0,\tau]}(x)$ such that
$$V(\tau, x) \le \frac12 \int_0^\tau \|u(\sigma)\|_U^2 \, \ud \sigma < V(\tau,x)+\varepsilon,$$
so that for the corresponding $\tilde{u}$ we get
$$V(t,x) \le \frac12 \int_0^t \|\tilde{u}(s)\|_U^2 \, \ud s = \frac12 \int_0^\tau \|u(\sigma)\|_U^2 \, \ud \sigma < V(\tau,x)+\varepsilon$$
and finally $V(t,x)\le V(\tau,x)$.\\[2mm]
(ii) Formula (\ref{valuef}) shows that $V$ is quadratic with respect to $x$.
Moreover (\ref{valuef}), rewritten in $H$, becomes
\begin{eqnarray*}
V(t,x) & = & \frac12 \langle Q_{\infty}^{1/2} Q_{t}^{-1/2}x,Q_{\infty}^{1/2} Q_{t}^{-1/2}x \rangle_H \\
& = & \frac12 \langle [Q_{\infty}^{1/2} Q_{t}^{-1/2}]^{*H} Q_{\infty}^{1/2} Q_{t}^{-1/2}x,x \rangle_{[R(Q_{t}^{1/2})]^{*H},R(Q_{t}^{1/2})} \quad \forall t>0, \ \ \forall x\in R(Q_{t}^{1/2}),
\end{eqnarray*}
and the claim is proved.\\[2mm]
(iii)-(a) Under Assumption \ref{NC} formula (\ref{forquaV2}) immediately follows from the fact that $R(Q_t^{1/2})=H$ for every $t\ge T_0$.
To prove (\ref{formaP2}) we take $x\in R(Q_t)$. Then, by (\ref{valuef}),
$$
V(t,x)= \frac12 \| Q_{t}^{-1/2}x \|_X^2
=\frac12 \langle Q_{t}^{-1/2}x, Q_{t}^{-1/2}x \rangle_X
= \frac12 \langle Q_{t}^{-1}x, x \rangle_X
$$
where in the last step we used Lemma \ref{relinH} (i).
Now passing to the inner product in $H$ and using Lemma \ref{relinH} (ii)
we get
$$
V(t,x)= \frac12 \langle [Q_{\infty}^{1/2} Q_{t}^{-1}]x, Q_{\infty}^{1/2}x\rangle_H
=\frac12 \langle[Q_{\infty} Q_{t}^{-1}]x, x\rangle_H \,,
$$
which is the claim.
Finally the estimate (\ref{glolim}) is an immediate consequence of
the monotonicity of $V(\cdot, x)$.\\[2mm]
(iii)-(b) First we show that if $t\in[T_0,\infty[\,$, then for each $\varepsilon>0$ and $R>0$ there is $\delta>0$ such that
\begin{equation}\label{eq:contV1}
V(\tau,x) < V(t,x) + \varepsilon \qquad \forall x\in H, \; \|x\|_H\le R, \quad \forall \tau\in [T_0\vee (t-\delta), t[\,.
\end{equation}
To this purpose, fix $\tau \in [T_0,t[\,$ and $x \in B_H(0,R)$, take $u\in {\cal U}_{[0,t]}(x)$ such that
$$\frac12 \int_0^t \|u(s)\|_U^2\, \ud s < V(t,x)+ \frac{\varepsilon}{2}\,,$$
and define
$$\overline{u}(s) = u(s+t-\tau), \qquad s\in [0,\tau].$$
Since
$$\int_0^\tau e^{(\tau-s)A}B\overline{u}(s)\,\ud s = \int_0^\tau e^{(\tau-s)A}Bu(s+t-\tau)\,\ud s = \int_{t-\tau}^t e^{(t-\sigma)A}Bu(\sigma) \,\ud \sigma= x - \int_0^{t-\tau}e^{(t-\sigma)A}Bu(\sigma)\,\ud \sigma,$$
we have $\overline{u}\in {\cal U}_{[0,\tau]}\left(x-\int_0^{t-\tau} e^{(t-\sigma)A}Bu(\sigma)\,\ud \sigma\right)$. Hence
$$V\left(\tau,x-\int_0^{t-\tau} e^{(t-\sigma)A}Bu(\sigma)\,\ud \sigma\right) \le \frac12 \int_0^\tau \|\overline{u}(s)\|_U^2 \,\ud s = \frac12 \int_{t-\tau}^t \|u(\sigma)\|_U^2 \ud \sigma \le \frac12 \int_0^t \|u(\sigma)\|_U^2\,\ud \sigma< V(t,x)+ \frac{\varepsilon}{2}\,.$$
On the other hand, by (\ref{forquaV2}) we have, for $x,y \in H$
\begin{eqnarray*}
V(\tau,x-y)-V(\tau,x) & = & \frac12 \langle P_V(\tau)(x-y),x-y\rangle_H - \frac12 \langle P_V(\tau)x,x \rangle_H \\
& = & - \frac12 \langle P_V(\tau)y,x-y\rangle_H - \frac12 \langle P_V(\tau)x,y \rangle_H;
\end{eqnarray*}
then, by (\ref{glolim}), for every $\eps>0$ and $R>0$ there is $\eta>0$ such that
$$\|y\|_H < \eta \qquad \Longrightarrow \qquad |V(\tau,x-y)-V(\tau,x)|< \frac{\varepsilon}{2} \qquad \forall \tau\ge T_0, \quad \forall x \in B_H(0,R).$$
Hence we get
$$V(\tau,x) <V\left(\tau,x-\int_0^{t-\tau} e^{(t-\sigma)A}Bu(\sigma)\,\ud \sigma\right)+ \frac{\varepsilon}{2}<V(t,x)+ \varepsilon$$
provided we are able to find $\delta>0$ such that
\begin{equation}\label{deltasigma}
\left\|\int_0^{t-\tau} e^{(t-\sigma)A}Bu(\sigma)\,\ud \sigma\right\|_H < \eta  \qquad \forall \tau\in [T_0\vee (t-\delta), t[\,.
\end{equation}
In order to check \myref{deltasigma}, we fix $z\in R(Q_\infty)$ with $\|z\|_H \le 1$. We can write, by Assumption \ref{NC}, Proposition \ref{Pr:B.1} and Lemma \ref{composQt},
\begin{eqnarray*}
\left|\left\langle \int_0^{t-\tau} e^{(t-\sigma)A}Bu(\sigma)\,\ud \sigma, z\right\rangle_H\right| & = & \left|\left\langle \int_0^{t-\tau} e^{(t-\sigma)A}Bu(\sigma)\,\ud \sigma, Q_\infty^{-1}z\right\rangle_X \right| \\
& = & \left|\left\langle \int_0^{t-\tau} e^{(t-\sigma-T_0)A}Bu(\sigma)\,\ud \sigma, e^{T_0 A^*}Q_\infty^{-1}z \right\rangle_X\right| \\
& \le & \int_0^{t-\tau} e^{(t-\sigma-T_0)A}Bu(\sigma)\|_X \,\ud \sigma \, \cdot \,\| e^{T_0 A^*}Q_\infty^{-1}z \|_X \\
& \le & \int_0^{t-\tau} e^{(t-\sigma-T_0)A}Bu(\sigma)\|_X\,\ud \sigma \, \cdot \,\|Q_{T_0}^{1/2}Q_\infty^{-1}z \|_X \\
& \le & c \int_0^{t-\tau} e^{-\omega(t-\sigma-T_0)}\|u(\sigma)\|_U\, \ud \sigma \, \cdot \,\|Q_\infty^{-1/2}z\|_X \\
& \le & c\,\sqrt{t-\tau}\, e^{-\omega\,\delta}\|u\|_{L^2(0,t;U)} \|z\|_H.
\end{eqnarray*}
Hence, using the density of $R(Q_\infty)$ in $H$ (see Lemma \ref{HH}),
$$\sup_{\|z\|_H \le 1} \left|\left\langle \int_{-t}^{-\tau} e^{-sA}Bu(s)\,\ud s, z\right\rangle_H \right| \le c\,\sqrt{t-\tau}\, \|u\|_{L^2(-t,0;U)},$$
so that
$$\left\|\int_{-t}^{-\tau} e^{-sA}Bu(s)\,\ud s \right\|_H \le c\,\sqrt{t-\tau}\, \|u\|_{L^2(-t,0;U)}.$$
Thus, to achieve (\ref{deltasigma}) it suffices to take $\delta>0$ such that $c\sqrt{\delta}\, \|u\|_{L^2(-t,0;U)}< \eta$. Hence we have proved (\ref{eq:contV1}), too.\\
Now fix $\eps >0$, $R>0$ and take $\delta$ such that (\ref{eq:contV1}) holds. For $(t,x), (\tau,x')\in [T_0,+\infty]\times B_H(0,R)$ with $|t -\tau|<\delta$ we have
\begin{eqnarray*}
|V(t,x)-V(\tau,x')| & \le & |V(t,x)-V(\tau,x)|+ |V(\tau,x)-V(\tau,x')| \\
& \le & \eps + \frac12 |\langle P_V(\tau)x,x \rangle_H - \frac12 \langle P_V(\tau)x',x' \rangle_H| \\
& \le & \eps + \frac12 |\langle P_V(\tau)(x-x'),x \rangle_H|+\frac12 |\langle P_V(\tau)x',x-x'\rangle_H| \\
& \le & \eps+ \|P_V(T_0\|_{\Lc(H)}R \|x-x'\|_H\,,
\end{eqnarray*}
and the first part of the claim easily follows.
To prove the continuity of the map $t \to P_V(t)$ we observe that, for $t,\tau\in [T_0,+\infty[$
$$
\|P_V(t)-P_V(\tau)\|_{\Lc(H)}=\sup_{\|x\|= 1} \langle P_V(t)-P_V(\tau)x,x \rangle_H
=2\sup_{\|x\|= 1} |V(t,x)-V(\tau,x)|;
$$
so the claim follows by (\ref{eq:contV1}).\\[2mm]
(iii)-(c) The limit in (\ref{Vinfty}) clearly exists and is finite by monotonicity and positivity of $V$.
To find this limit we consider first the case when $x\in R(Q_\infty)$. Then we have, by (\ref{formaP2}),
\begin{eqnarray*}
2V(t,x)-\|x\|_H^2 & = & \langle P_V(t)x-x,x \rangle_H = \langle Q_\infty Q_t^{-1}x-x,x\rangle_H \\
& = & \langle (Q_\infty -Q_t)Q_t^{-1}x,x \rangle_H =\langle(Q_\infty -Q_t)Q_t^{-1}x,Q_\infty^{-1}x \rangle_X\,.
\end{eqnarray*}
Since, for suitable $c>0$,
$$\|(Q_\infty -Q_t)z\|_X=\left\|\int_t^\infty e^{sA}BB^*e^{sA^*}z\,\ud s\right\|_X
\le c \, e^{- 2\omega t}\|z\|_X \quad \forall z\in X,$$
we obtain
$$2V(t,x)-\|x\|_H^2 \le c \, e^{- 2\omega t}\| Q_t^{-1}x\|_X \| Q_\infty^{-1}x\|_X.$$
But
\begin{eqnarray*}
\| Q_t^{-1}x\|_X & = & \|Q_\infty^{1/2} Q_t^{-1}x\|_H
=\|Q_\infty^{1/2} Q_t^{-1/2} Q_t^{-1/2}Q_\infty^{1/2}Q_\infty^{-1/2} x\|_H \\[1mm]
& \le & \|Q_\infty^{1/2} Q_t^{-1/2}\|_{\mathcal{L}(H)} \| Q_t^{-1/2}Q_\infty^{1/2}\|_{\mathcal{L}(H)}\|Q_\infty^{-1/2}x\|_H,
\end{eqnarray*}
so that by Lemma \ref{composQt} we get
$$\lim_{t\to \infty} V(t,x) =\frac12\|x\|_H^2 \qquad \forall x\in R(Q_\infty).$$
By selfadjointness of $P_V(t)$ and polarization, we also have
$$\lim_{t\to \infty} \langle P_V(t)x,y \rangle_H = \langle x,y \rangle_H \quad \forall x,y \in R(Q_\infty).$$
Since, by (\ref{glolim}), $P_V(t) - I$ is uniformly bounded, by density (Lemma \ref{HH})
we deduce that
$$\lim_{t\to \infty} \langle P_V(t)x,y \rangle_H = \langle x,y \rangle_H \quad \forall x\in R(Q_\infty), \quad \forall y \in H,$$
and using again that $P_V(t)$ is selfadjoint we get
$$\lim_{t\to \infty} \langle x,P_V(t)y \rangle_H = \langle x,y \rangle_H  \quad \forall x\in R(Q_\infty), \quad \forall y \in H.$$
With the same argument we then obtain
$$\lim_{t\to \infty} \langle x,P_V(t)y \rangle_H = \langle x,y \rangle_H \quad \forall x,y\in H,$$
and the result follows.\hfill \qed

\end{document}